\documentclass[10pt]{amsart}
\usepackage{amsfonts,amsmath,amssymb,color,amscd,amsthm}
\usepackage[T1]{fontenc} 
\usepackage[francais,english]{babel} 
\usepackage[all]{xy}
\usepackage{tikz-cd}
\usepackage{pgf,tikz}
\usepackage{pgfplots}
\usepgfplotslibrary{patchplots}
\usepackage{soul}
\pgfplotsset{
        compat=1.11,
        My Line Style/.style={
            smooth,
            thick,
            samples=400,
        },
    }
\usepackage{stmaryrd}

\usepackage{enumitem} 
\usepackage[backref, colorlinks, linktocpage, citecolor = blue, linkcolor = blue]{hyperref} %permet d'afficher les r\'ef\'erences

\newcommand\iso{\stackrel{\simeq}{\longrightarrow}}

\newcommand\Jac{\mathrm{Jac}}
\newcommand\car{\mathrm{char}}
\newcommand\N{{\mathbb N}}

\newcommand\R{{\mathbb R}}
\newcommand\A{{\mathbb A}}
\newcommand\Z{{\mathbb Z}}
\newcommand\C{{\mathbb C}}

\newcommand\tr{\hbox to 1mm  {${}^t \!  $} }

\newcommand\p{{\mathbb P}}

\renewcommand\k{\mathbf{ \textbf k}}

\DeclareMathOperator{\Aut}{Aut}

\DeclareMathOperator{\Aff}{Aff}
\DeclareMathOperator{\GL}{GL}

\DeclareMathOperator{\SL}{SL}

\DeclareMathOperator{\Spec}{Spec}

\DeclareMathOperator{\Bl}{B\ell}

\DeclareMathOperator{\pr}{pr}
\DeclareMathOperator{\Gal}{Gal}

\newtheorem{theorem}{Theorem}
\newtheorem{lemma}{Lemma}[section]
\newtheorem{corollary}[lemma]{Corollary}
\newtheorem{proposition}[lemma]{Proposition}
\newtheorem{question}[lemma]{Question}

\newtheorem*{corollary*}{Corollary}
\newtheorem*{corollary**}{Corollary}

\theoremstyle{definition}
\newtheorem{definition}[lemma]{Definition}

\theoremstyle{remark}
\newtheorem{remark}[lemma]{Remark}
\newtheorem{example}[lemma]{Example}

\title{Embeddings of Affine Spaces into Quadrics} % This is the full title of the paper

\author{J\'er\'emy Blanc}
\address{J\'er\'emy Blanc, Universit\"{a}t Basel, Departement Mathematik und Informatik, Spiegelgasse $1$, CH-$4051$ Basel, Switzerland}
\email{jeremy.blanc@unibas.ch}

\author{Immanuel van Santen (n\'e Stampfli)}
\address{Immanuel van Santen,
Fachbereich Mathematik der Universit\"{a}t Hamburg,
Bundesstra{\ss}e $55$, DE-$20146$ Hamburg, Germany}
\email{immanuel.van.santen@math.ch}

\begin{document}

\maketitle

\begin{abstract}

This article provides, over any field, infinitely many algebraic embeddings of the 
affine spaces $\mathbb{A}^1$ and $\mathbb{A}^2$ 
into smooth quadrics of dimension two and three respectively, which are pairwise non-equivalent under automorphisms of the smooth quadric.
Our main tools are the study 
of the birational morphism $\mathrm{SL}_2 \to \mathbb{A}^3$
and the fibration
$\mathrm{SL}_2 \to \mathbb{A}^3 \to \mathbb{A}^1$ obtained by projections, 
as well as degenerations of variables of polynomial rings, 
and families of $\mathbb{A}^1$-fibrations.
\end{abstract}

\tableofcontents

\section{Introduction} In the sequel we denote by $\k$ the ground field of our algebraic varieties.
Given two affine algebraic varieties $X,Y$, we say that two closed embeddings $\rho,\rho'\colon X\hookrightarrow Y$ are \emph{equivalent} if there exists an automorphism $\varphi\in \Aut(Y)$ such that $\rho'=\varphi \circ \rho$. Similarly, we say that two closed subvarieties $X,X'\subset Y$ are \emph{equivalent} if there exists an automorphism $\varphi\in \Aut(Y)$ such that $X'=\varphi(X)$. If two closed embeddings are equivalent, 
then their images are equivalent, but the converse is not always true and is related to the extension of automorphisms.

\medskip

In the Bourbaki Seminar {\it Challenging problems on affine $n$-space} 
\cite{KraftChallenging}, Hans\-peter Kraft gives a list of eight 
fundamental problems related to the affine $n$-spaces. The third one is the following:
\begin{center}{
 \textbf{Embedding Problem.} \it Is every closed embedding $\A^m_\k\hookrightarrow \A^n_\k$ equivalent to the standard embedding $(x_1,\dots,x_m)\mapsto (x_1,\dots,x_m,0,\dots,0)$?}\end{center}
This question, asked over the ground field $\k=\mathbb{C}$ in \cite{KraftChallenging}, has until now no negative answer. For $\k=\mathbb{R}$, it is easy to find counterexamples for $m=1$ and $n=3$, by taking embeddings which are not topologically trivial (non-trivial knots), 
see for instance the example of \cite{Shastri92}, reproduced below in 
Example~\ref{Ex:Shastri}. 
In positive characteristic, there are counterexamples when $m=n-1$ 
(see Proposition~\ref{Prop:EmbeddingsCharP}). The embedding problem has however a positive answer in the following cases:
\begin{enumerate}
\item
$m=1$, $n=2$, $\car(\k)=0$ (Abhyankar-Moh-Suzuki Theorem) \cite{AbhyankarMoh, Suzuki}, \cite[Theorem~2.3.5]{Essen2000}; 
\item
$n\ge 2m+2$, $\k$ infinite (Theorem of Kaliman, Nori and 
Srinivas \cite{Kaliman, Srinivas}).
\end{enumerate}

The case of hypersurfaces $(m=n-1)$ is of particular interest. 
In this case, the image is given by the zero set of an irreducible polynomial equation $f\in \k[\A^n]$. One necessary condition 
for an embedding to be equivalent to the standard embedding 
consists of asking that the other fibres of $f \colon \A_{\k}^n \to \A_{\k}^1$
are affine spaces. In fact, for any field $\k$ and any $n\ge 1$, there is no known example of a hypersurface $X\subset \A_{\k}^n$ isomorphic to $\A_{\k}^{n-1}$ and given by $f=0$, $f\in \k[\A^n]$ irreducible such that another fibre $f=\lambda$ is not isomorphic to $\A_{\k}^{n-1}$. It is conjectured by Abhyankar and
Sathaye that no such examples exist, at least when $\car(\k)=0$, 
see \cite[\S3, page 103]{Essen2000}, even if this is quite strong and seems ``unlikely'' (as Arno van den Essen says in \cite[\S3, page 103]{Essen2000}). Moreover, for $n=3$ and $\car(\k) = 0$, the fact that infinitely many 
fibres $f=\lambda$ are isomorphic to $\A^2_\k$ implies that the fibration is equivalent to the standard one, and in particular that all fibres are isomorphic to $\A^2_\k$ \cite{KalimanZaidenberg01,Kaliman02, DaigleKaliman09}. 
For $\car(\k)>0$, there is until now no known counterexample to the above conjecture, which is open even in dimension $n=2$ (and corresponds to a question of Abhyankar, see \cite[Question 1.1]{Ganong2011}).

\bigskip

In this paper, we replace the affine space at the target by some analogue varieties, namely affine smooth quadrics. This simplifies the question in such a way that one can actually give an answer. Moreover, it also gives some idea on what kind of behaviour one could expect in a general situation.

\bigskip

In dimension $n=2$, the most natural quadric is 
\[
	Q_2=\Spec(\k[x,y,z]/(xy-z(z+1))) \subset \A^3_\k.
\] 
In fact, if $\k$ is an algebraically closed field, then every smooth quadric hypersurface $Q\subset \A^3_\k$ is isomorphic to $\A^2_\k$, 
$(\A_{\k}^1 \setminus \{ 0 \}) \times \A^1_\k$ or $Q_2$, as one can see 
using the classification
of quadratic forms. As all embeddings of 
$\A^1_\k$ into $(\A_{\k}^1 \setminus \{ 0 \}) \times \A^1_\k$ are constant on the first factor,
they are all equivalent.
Over any field, the group of automorphisms of $Q_2$ is similar to the one of $\A^2_\k$, as it is an amalgamated product of two factors, corresponding to affine maps and triangular maps \cite[Theorem~5.4.5(7)(a)]{BlaDub}. This is also the case for the affine surface $\p^2_{\k} \setminus \Gamma$, where $\Gamma\subset \p_{\k}^2$ 
is any smooth conic having a $\k$-point (see for instance \cite[Theorem~2]{DecDub}). If $\car(\k)=0$, there is exactly one (respectively two) closed 
curve $C\subset \A^2_\k$ (respectively $C\subset \p_{\k}^2\setminus \Gamma$) isomorphic to $\A^1_\k$, up to automorphism of the surface. This follows from 
the Abhyankar-Moh-Suzuki Theorem for $\A^2_\k$ and from \cite{DecDub} for 
$\p_{\k}^2\setminus \Gamma$. In particular, all automorphisms of the corresponding curves extend to automorphisms of $\A^2_\k$ or $\p_{\k}^2\setminus \Gamma$.
Similarly, a complex toric affine surface admits only finitely many embeddings of 
$\A^1_\C$, up to equivalence \cite{ArzZai}. 
By contrast, we prove the following result:

\begin{theorem}\label{Thm:EmbQ2}
For each field $\k$, there is an infinite set of closed curves \[C_i\subset Q_2=\Spec(\k[x,y,z]/(xy-z(z+1))),\ i\in I,\]
which are pairwise non-equivalent up to automorphism of $Q_2$, such that each $C_i$ is isomorphic to $\A^1_\k$ and such that the identity is the only automorphism of $C$ that extends to an automorphism of $Q_2$.
Moreover, if $\k$ is uncountable, then one can choose the same for $I$.
\end{theorem}

\medskip

In dimension $n=3$,  the most natural quadric is
\[
	\SL_2=\Spec(\k[t,u,x,y] / (xy-tu-1))\subset \A_{\k}^4 \, .
\]
Similarly as in dimension two, over an algebraically closed field $\k$,
every quadric hyper\-surface in $\A^4_{\k}$
is isomorphic to $\A^3_{\k}$, $(\A_{\k}^1 \setminus \{ 0 \}) \times \A^2_{\k}$, 
$Q_2 \times \A^1_{\k}$ or $\SL_2$. Moreover, the quotient of $\SL_2$ by its maximal torus yields a morphism $\SL_2\to\SL_2/T\simeq Q_2$, which is the ``universal torsor'' (also called the Cox quotient presentation or the characteristic space), see \cite[Examples 4.5.13--4.5.14]{ArDHL}.

We consider the quadric hypersurface $\SL_2$ more closely.
Its automorphism group shares similar properties with the one of $\Aut(\A^3_\k)$  (see \cite{LamyVenereau, BisiFurterLamy, Martin}). Both are known to be complicate, as they contain ``wild'' automorphisms \cite{LamyVenereau}, and do not preserve any fibration, as it is the case for other varieties being topologically closer to $\A^3_\k$, like the Koras-Russell threefold. However, in contrast to the quadric
$Q_2$, the quadric $\SL_2$ is closer to a contractible affine variety
in the sense that the ring of regular functions on $\SL_2$ 
is a unique factorisation domain (see Lemma~\ref{lem:R_is_UFD}). 
The first difference concerning embeddings
of affine spaces with the surfaces $Q_2, \A^2_\k, \p^2_\k\setminus \Gamma$ and with $\A^3_\k$ is that the ``simplest embedding'' $\A^2_\k\hookrightarrow \SL_2$ is more rigid in the following sense:

\begin{theorem}\label{Thm:StdEmbSL2}
Let $\k$ be any field and let \[
\begin{array}{cccc}
\rho_1\colon&\A^2_\k&\hookrightarrow& \SL_2\\
&(s,t)&\mapsto &\begin{pmatrix}
				  1 & t \\
				  s & 1+st
			     \end{pmatrix}\end{array}\] be the ``standard'' embedding. Then, an automorphism $(s,t)\mapsto (f(s,t),g(s,t))$ of $\A^2_\k$ extends to an automorphism of $\SL_2$, via $\rho_1$, if and only if it has Jacobian determinant $\frac{\partial f}{\partial s}\frac{\partial g}{ \partial t}-\frac{\partial f}{\partial t}\frac{\partial g}{\partial s}\in \k^*$ 
equal to $\pm 1$. In particular, the following holds:
\begin{enumerate}[label=$(\arabic*)$, ref=\arabic*]
	\item \label{Thm:StdEmbSL2ass1} every embedding 
	$\A_{\k}^2 \hookrightarrow \SL_2$ with image
		$\rho_1(\A^2_{\k})$ is equivalent to an embedding
		\[
			\begin{array}{cccc}
			\rho_\lambda \colon &\A^2_\k&\hookrightarrow& \SL_2\\
			&(s,t)&\mapsto &\begin{pmatrix}
				  1 & t \\
				  \lambda s & 1+\lambda st
			\end{pmatrix},\end{array}
		\]
		for a certain $\lambda \in \k^\ast$. Moreover, 
		$\rho_{\lambda}$ and $\rho_{\lambda'}$ are equivalent if
		and only if $\lambda' = \pm \lambda$;
	\item \label{Thm:StdEmbSL2ass2}
	if $\k$ has at least $4$ elements, 
		then not all automorphisms of $\A^2_\k$ extend to $\SL_2$ via $\rho_1$.
\end{enumerate}
\end{theorem} 
\begin{remark}
\label{Rem:HolSmall}
Let us make some comments on 
Theorem~\ref{Thm:StdEmbSL2}:
\begin{enumerate}\item
Over the field of complex numbers $\k = \C$,
we show that all algebraic automorphisms of $\A_{\k}^2$ extend
via the standard embedding $\rho_1$ to holomorphic automorphisms of $\SL_2$,
see Remark~\ref{Rem:HolomorphicA2inSL2}.
\item
If all component functions of a closed
embedding $f \colon \A_{\k}^2 \hookrightarrow \SL_2$ 
are polynomials of degree $\leq 2$, then $f$ is equivalent to $\rho_\lambda$
for a certain $\lambda \in \k^\ast$ (Proposition~\ref{Prop:Low_degree}).
\end{enumerate}\end{remark}

Next, we focus on the closed embeddings $\A^2_\k \hookrightarrow \SL_2$ 
that are compatible with the simplest $\A^2$-fibration of $\SL_2$. More precisely:

\begin{definition}
A closed embedding $\rho\colon \A^2_\k\hookrightarrow \SL_2$ is said to be a \emph{fibred embedding} if it is of the form
 \[\label{Eq:FibredEmb}
	\tag{$\diamondsuit$}
\begin{array}{cccc}
\rho\colon&\A^2_\k&\hookrightarrow& \SL_2\\
&(s,t)&\mapsto &\begin{pmatrix}
				  p(s,t) & t \\
				  r(s,t) & q(s,t)
			     \end{pmatrix}\end{array}\]
			     for some $p,q,r\in \k[s,t]$. This corresponds to the commutativity of the diagram \[\xymatrix@R=4mm@C=2cm{\A^2_\k\ar[rd]_{\pi_1}\ar@{^{(}->}^{\rho}[rr] &&\SL_2\ar[ld]^{\pi_2} \\
			     & \A^1_\k,}\]
			     where $\pi_1\colon \A^2_\k\to \A^1_\k$, $\pi_2\colon \SL_2\to \A^1_\k$ are respectively given by $(s,t)\mapsto t$ and $\begin{pmatrix}
				  x & t \\
				  u & y
			     \end{pmatrix}\mapsto t$.
			     \end{definition}
			     
As we will show, there are a lot of fibred embeddings (i.e.~embeddings of the 
form~\eqref{Eq:FibredEmb}):

\begin{theorem} 
\label{thm:Fibered_Embeddings}
Let $\k$ be any field, let $P\in \k[t,x,y]$ be a polynomial that is a variable of the $\k(t)$-algebra $\k(t)[x,y]$ $($which means that $P$ is the image of $x$ by some automorphism of the 
$\k(t)$-algebra $\k(t)[x,y]$$)$, and let $H_P\subset \SL_2=\Spec(\k[t,u,x,y] / (xy-tu-1))$ and $Z_P\subset \A^3_\k=\Spec(\k[t,x,y])$ be the hypersurfaces given by $P=0$.
\begin{enumerate}[label=$(\arabic*)$, ref=\arabic*]
\item\label{ThmFibered1}
The following conditions are equivalent:
\begin{enumerate}[label=$(\alph*)$, ref=\alph*]
\item\label{ThmEqa}
The hypersurface $H_P\subset \SL_2$ is isomorphic to $\A^2_\k$.
\item\label{ThmEqb}
The hypersurface $H_P\subset \SL_2$ is the image of a fibred embedding $\A^2_\k\hookrightarrow \SL_2$.
\item\label{ThmEqc}
The fibre
of $Z_P \to \A^1_{\k}$, $(t, x, y) \mapsto t$ over every closed point 
of $\A_{\k}^1\setminus \{0\}$
is isomorphic to $\A^1$ and the 
polynomial $P(0,x,y)\in \k[x,y]$ is of the form  $\mu x^m(x-\lambda)$ or 
			     $\mu y^m(y-\lambda)$  
			     for some $\mu,\lambda\in \k^*$  
			     and some $m\ge 0$. 
			
\end{enumerate}
			     \item\label{ThmFibered2}
If $P,Q\in \k[t,x,y]$ are two polynomials of the above form satisfying the conditions $\eqref{ThmEqa}-\eqref{ThmEqb}-\eqref{ThmEqc}$, such that $H_P,H_Q\subset \SL_2$ are equivalent under an automorphism of $\SL_2$, then $Z_P,Z_Q\subset \A^3_\k$ are equivalent under an automorphism of $\A^3_\k$.
\item\label{ThmFibered3}
There are infinitely many fibred embeddings $\A^2_\k\hookrightarrow \SL_2$ 
having pairwise non-equivalent images in $\SL_2$. If $\k$ is uncountable, we can moreover choose uncountably many such embeddings.
\end{enumerate}
\end{theorem}

\begin{remark}Let us make some comments on 
Theorem~\ref{thm:Fibered_Embeddings}:
\begin{enumerate}
\item
It is possible that $H_P, H_Q$ are non-equivalent, 
even if $Z_P,Z_Q$ are equivalent (Lemma~\ref{Lemm:NonEquivZ}).
\item
If $\car(\k)=0$, then every image of a fibred embedding $\A^2_\k\hookrightarrow \SL_2$ is of the form $H_P$ as above (Lemma~\ref{Lemm:Polynomialp}\eqref{DegenerateFibre}). 
This is false if $\car(\k)>0$ (Lemma~\ref{Lemm:ExampleNotVarkt}).
\end{enumerate}
\end{remark}

Let us make the following comment concerning
embeddings of $\A^1_{\k}$ into the smooth quadric $\SL_2$
over the field $\k = \C$. 
Although there are infinitely many
non-equivalent embeddings of $\A^2_{\C}$ into $\SL_2$,
it is not known, whether all embeddings of $\A^1_{\C}$ into $\SL_2$
are equivalent under an algebraic automorphism. It seems
that this questions is as difficult, as the question of
non-equivalent embeddings 
$\A^1_{\C} \hookrightarrow \A^3_{\C}$. However, up to holomorphic
automorphisms, all embeddings of $\A^1_{\C}$ into $\A^3_{\C}$
and into $\SL_2$ are equivalent, see \cite{Kaliman92, Stampfli15}. 

In the last section (Lemma~\ref{Exam:RSL2R}), we give an example of 
an embedding $\A^1_{\R} \hookrightarrow \SL_2$ which is
non-equivalent to the standard embedding.

\subsection*{Acknowledgement}
We would like to thank Peter Feller for many fruitful discussions, and Ivan Arzhantsev for indicating us some references.

\section{The smooth quadric of dimension $2$ and the proof of Theorem~\ref{Thm:EmbQ2}}
\subsection{The isomorphism with the complement of the diagonal in $\p^1_\k\times \p^1_\k$}
In this section, we study the smooth quadric $Q_2\subset \A^3_\k$ given by
\[Q_2=\Spec(\k[x,y,z]/(xy-z(z+1))),\]
and more particularly closed embeddings $\A^1_\k\hookrightarrow Q_2$. Since the closure of $Q_2$ in $\p^3_\k$ is a smooth quadric, isomorphic to $\p^1_\k\times \p^1_\k$, we get the following classical isomorphism:
\begin{lemma}\label{Lem:IsoRho} The morphism
 \[\begin{array}{cccc}
\rho\colon&Q_2&\to& \p^1_\k\times \p^1_\k\\
&(x,y,z)&\mapsto &\left\{\begin{array}{ccccccc}
(&[y:z]&,&[z:x]&)& \text{ if }z\not=0\\
(&[z+1:x]&,&[y:z+1]&  )&\text{ if }z\not=-1
\end{array}\right.\end{array}\]
yields an isomorphism $Q_2\iso (\p^1_\k\times \p^1_\k)\setminus \Delta$, where $\Delta\subset \p^1_\k\times \p^1_\k$ is the diagonal, with an inverse given by 
\[\begin{array}{cccc}
\psi\colon& (\p^1_\k\times \p^1_\k)\setminus \Delta&\to& Q_2\\
&([u_0:u_1],[v_0:v_1])&\mapsto &\left(\frac{u_1v_1}{u_0v_1-u_1v_0},\frac{u_0v_0}{u_0v_1-u_1v_0},\frac{u_1v_0}{u_0v_1-u_1v_0}\right)\end{array}\]
\end{lemma}
\begin{proof} We first check that $\rho((x,y,z))\in (\p^1_\k\times \p^1_\k)\setminus \Delta$ for each $(x,y,z)\in Q_2$. If $z\not=0$ then $[y:z]\not=[z:x]$, since $xy-z^2=z \not=0$. If $z=0$, then $xy=0$, whence $[z+1:x]=[1:x]\not=[y:1] =[y:z+1]$.

It remains then to check that $\rho\circ \psi=\mathrm{id}_{(\p^1_\k\times \p^1_\k)\setminus \Delta}$ and $\psi\circ \rho=\mathrm{id}_{Q_2}$, which follows from a straight-forward calculation.
\end{proof}
%\begin{proposition}\cite[Theorem~5.4.5(7)(a)]{BlaDub}\label{Prop:P1P1Diag}
%Let $\Delta\subset \p^1_\k\times \p^1_\k$ be the diagonal and let $\pi\colon (\p^1_\k\times \p^1_\k) \setminus \Delta\to \A^1_\k$ be the morphism given by $([u_0:u_1],[v_0:v_1])\mapsto \frac{u_1v_1}{u_0v_1-u_1v_0}$. 
%The group $\Aut((\p^1_\k\times \p^1_\k) \setminus \Delta)$ is the amalgamated product of the two groups 
%\[\begin{array}{rcl}
%A&=&\{g\in \Aut(\p^1_\k\times \p^1_\k)\mid g(\Delta)=\Delta\}\\
%J&=&\Aut((\p^1_\k\times \p^1_\k) \setminus \Delta,\pi)=\{g\in \Aut((\p^1_\k\times \p^1_\k) \setminus \Delta)\mid \exists h\in \Aut(\A^1_\k), h\pi=\pi g\}\end{array}\]
%$($this means that $\Aut((\p^1_\k\times \p^1_\k) \setminus \Delta)$ is generated by $A$ and $J$ and is the quotient of the free group $A\star J$ by the relation induced by $a bc$, $a,b,c\in A\cap J$, $a\cdot b\cdot c=1$ in $A\cap J)$.
%\end{proposition}

\subsection{Families of embeddings}
The following  result is the key step in the proof of Theorem~\ref{Thm:EmbQ2}. 
\begin{lemma}\label{Lem:nup}\item
\begin{enumerate}[label=$(\arabic*)$, ref=\arabic*]
\item\label{nupclosedemb}
For each polynomial $p\in \k[t]$, the morphism $\nu_p\colon \A^1_\k\hookrightarrow Q_2$ given by
 \[\begin{array}{cccc}
\nu_p\colon&\A^1_\k&\to& Q_2\\
&t&\mapsto &(t(1+tp(t)), p(t), tp(t))\end{array}\]
is a closed embedding. 
\item\label{ExtensionAuto}
If $p,q\in \k[t]$ are polynomials of degree $\ge 3$ 
such that $\alpha \nu_p=\nu_q\beta$ for some $\beta\in \Aut(\A^1_\k)$ and $\alpha\in \Aut(Q_2)$, then there exist $\mu\in \k$ and $\lambda \in \k^\ast$ 
such that
\[
	p(t)=\lambda q(\lambda t+\mu) \, , \ \ 
	\beta(t)=\lambda t+\mu \, , \ \ 
	\alpha(x,y,z)=
	\left(\lambda x+\frac{\mu^2}{\lambda} y+2\mu z+\mu, \frac{y}{\lambda}, 
	z+\frac{\mu}{\lambda} y\right) \, .
\]

\end{enumerate}
\end{lemma}
\begin{proof}
Using the isomorphism $\rho \colon Q_2\iso (\p^1_\k\times \p^1_\k)\setminus \Delta$ of Lemma~$\ref{Lem:IsoRho}$, we obtain that $\rho\circ \nu_p\colon \A^1_\k\to \p^1_\k\times \p^1_\k$ is given by $t\mapsto ([1:t],[p(t):1+tp(t)])$, which is the restriction of the closed embedding \[\hat{\nu}_p\colon \p^1_\k\hookrightarrow \p^1_\k\times \p^1_\k, [u:v]\mapsto ([u:v],[uP(u,v):u^{d+1}+vP(u,v)]),\] where $d=\deg(p)$ and $P(u,v)=p(\frac{v}{u})u^d$ is the homogenisation of $p$. This implies that $\Gamma_p=\hat{\nu}_p(\p^1_\k)\subset \p^1_\k\times \p^1_\k$ is a smooth closed curve (isomorphic to $\p^1_\k$), and since $\Gamma_p\cap \Delta$ is given by 
$u(u^{d+1}+vP(u,v))-vuP(u,v)=0$, i.e.~$u^{d+2}=0$, this shows that $\nu_p$ is a closed embedding, and thus yields (\ref{nupclosedemb}). 

It remains to prove Assertion~(\ref{ExtensionAuto}). 
We fix two polynomials $p,q\in \k[t]$ of degree $\ge 3$ such that $\alpha \nu_p=\nu_q\beta$ for some $\beta\in \Aut(\A^1_\k)$ and $\alpha\in \Aut(Q_2)$. This implies in particular that the automorphism $\alpha'=\rho^{-1}\alpha\rho\in \Aut((\p^1_\k\times \p^1_\k)\setminus \Delta)$ sends $\Gamma_p\setminus \Delta$ onto $\Gamma_q\setminus \Delta$.

We first prove that  $\alpha'\in \Aut((\p^1_\k\times \p^1_\k)\setminus \Delta)$ extends to an automorphism $\hat\alpha\in  \Aut(\p^1_\k\times \p^1_\k)$. Assume for contradiction that this is not the case. The map $\alpha'$ would then extend to a birational map $\hat{\alpha}\colon \p^1_\k\times \p^1_\k\dasharrow \p^1_\k\times \p^1_\k$, which is not an automorphism. We consider the minimal resolution of $\hat\alpha$, which is
\[
	\xymatrix@R=4mm@C=2cm{& Z\ar[ld]_{\chi_1}\ar[rd]^{\chi_2} \\
\p^1_\k\times \p^1_\k\ar@{-->}[rr]^{\hat\alpha} && \p^1_\k\times \p^1_\k\\
 (\p^1_\k\times \p^1_\k)\setminus \Delta \ar@{^{(}->}[u]\ar^{\simeq}[rr]&&    (\p^1_\k\times \p^1_\k)\setminus \Delta\ar@{^{(}->}[u] }
\]
where $\chi_1$, $\chi_2$ are birational morphisms. The resolution being minimal, every $(-1)$-curve $E\subset Z$ contracted by $\chi_2$ is not contracted by $\chi_1$, so $\chi_1(E)\subset \p^1_\k\times \p^1_\k$ is contracted by $\hat{\alpha}$. There is thus a unique $(-1)$-curve contracted by $\chi_2$, which is the strict transform $\tilde{\Delta}$ of $\Delta$, and satisfies $\chi_1(\tilde\Delta)=\Delta$. As $\Delta^2=2$ and $(\tilde{\Delta})^2=-1$, there are exactly three base-points of $\chi_1^{-1}$ that lie on the curve $\Delta$ (as proper point or infinitely near points).  Since $\Gamma_p$ is smooth of bidegree $(1,1+\deg p)$, we get 
$\Gamma_p\cdot \Delta=2+\deg p\ge 5$, which implies that the strict transforms of $\Gamma_p$ and $\Delta$ on $Z$ satisfy $\tilde{\Gamma}_p\cdot \tilde\Delta\ge 2$ (as only three points belonging to $\Delta$ have been blown-up). As the curve $\tilde\Delta$ is contracted by $\chi_2$, the curve $\chi_2(\tilde{\Gamma}_p)$ is singular. This contradicts the equality $\chi_2(\tilde{\Gamma}_p)=\Gamma_q$, which follows from the fact that $\hat\alpha(\Gamma_p\setminus \Delta)=\Gamma_q\setminus \Delta$.

We have shown that the extension of $\alpha'=\rho^{-1}\alpha\rho\in \Aut((\p^1_\k\times \p^1_\k)\setminus \Delta)$ is an automorphism $\hat{\alpha}\in \Aut(\p^1_\k\times \p^1_\k)$, which satisfies therefore $\hat{\alpha}(\Delta)=\Delta$ and $\hat{\alpha}(\Gamma_p)=\Gamma_q$. The curves $\Gamma_p$ and $\Gamma_q$ being of bidegree $(1,1+\deg p)$ and $(1,1+\deg q)$, we get $\deg p=\deg q$ and we obtain that $\hat\alpha$ does not exchange the two factors of $\p^1_\k\times \p^1_\k$. Moreover, as the point $([0:1],[0:1])=\Delta \cap \Gamma_p=\Delta\cap \Gamma_q$ is fixed, and as 
the diagonal $\Delta$ is invariant, we can write $\hat\alpha$ as
\[\hat\alpha([u_0:u_1],[v_0:v_1])=([u_0:\lambda u_1+\mu u_0],[v_0:\lambda v_1+\mu v_0]),\] for some $\lambda\in \k^*$, $\mu\in \k$. 

The equality $\alpha \nu_p=\nu_q\beta$ implies that $\hat\alpha\hat{\nu}_p=\hat{\nu}_q \hat\beta$, for some automorphism $\hat\beta\in \Aut(\p^1_\k)$, which is the extension of $\beta$ and therefore it is 
of the form $[u:v]\mapsto [u:\lambda v+\mu u]$. We then compute
\[
	\begin{array}{llllll}
	\hat\alpha\hat{\nu}_p([u:v])=
	([u:\lambda v+\mu u],[uP(u,v):\lambda u^{d+1}+\lambda vP(u,v)+\mu uP(u,v)])\\
	\hat{\nu}_q \hat\beta([u:v])= 
	([u:\lambda v+\mu u],[uQ(u,\lambda v+\mu u):
	u^{d+1}+(\lambda v+\mu u)Q(u,\lambda v+\mu u)])
	\end{array}
\]
and obtain that $P(u,v)=\lambda Q(u, \lambda v+\mu u)$. Remembering that $P(u,v)=p(\frac{v}{u})u^d$ and $Q(u,v)=q(\frac{v}{u})u^d$ we obtain that $p(t)=\lambda q(\lambda t+\mu)$. We then compute the explicit form of $\alpha$ by conjugating $\hat\alpha$ with $\rho^{-1}$.
\end{proof}
\begin{example}
\label{exa:Non-equivalent_Embeddings}
For each $n\ge 1$, let   $p_n(t) = t^n(t+1)^{n+1}$.
The closed curve $C_n=\nu_{p_n}(\A^1_\k)\subset Q_2$ is 
isomorphic to $\A^1_\k$, via
 \[
 	\begin{array}{cccc}
		\nu_{p_n}\colon&\A^1_\k&\hookrightarrow& Q_2\\
		&t&\mapsto &(t(1+t p_n(t)), p_n(t), t p_n(t)).
	\end{array}
\] 
Then Lemma~\ref{Lem:nup}(\ref{ExtensionAuto}) shows that all curves $C_n$ are non-equivalent for different $n \geq 1$, 
and that the identity is the only automorphism of $C_n$ that extends to $Q_2$.
\end{example}
The proof of Theorem~\ref{Thm:EmbQ2} is now a consequence of Lemma~\ref{Lem:nup}.
\begin{proof}[Proof of Theorem~$\ref{Thm:EmbQ2}$]
If $\k$ is the field with two elements, then we conclude by
Example~\ref{exa:Non-equivalent_Embeddings}.
Hence we can assume that $\k$ contains more than two elements.
For each $n\ge 1$ and each $\lambda\in \k$, 
$\varepsilon \in \k \setminus \{ 0, 1 \}$,
one defines $p_{n,\varepsilon}(t)=t^n(t+1)^{n+1}(t + \varepsilon)^{n+2}\in \k[t]$, and let $C_{n,\varepsilon}\subset Q_2$ be the closed curve given by $\nu_{p_{n,\varepsilon}}(\A^1_\k)$, which is isomorphic to $\A^1_\k$ (Lemma~\ref{Lem:nup}(\ref{nupclosedemb})). 

Lemma~\ref{Lem:nup}(\ref{ExtensionAuto}) implies that the identity is the only automorphism of $C_{n,\varepsilon}$ that extends to an automorphism of $Q_2$, since $\lambda p_{n,\varepsilon}(\lambda t+\mu)\not=p_{n,\varepsilon}(t)$, for
$(\lambda, \mu) \in (\k^\ast \times\k) \setminus \{ (1, 0) \}$.

Similarly, Lemma~\ref{Lem:nup}(\ref{ExtensionAuto}) shows that $C_{n,\varepsilon}$ is  equivalent to $C_{n',\varepsilon'}$ if and only if $n=n'$ and $\varepsilon=\varepsilon'$.
\end{proof}

%\begin{remark}
%If $\k$ is infinite, we can also find families of examples of the same degree, by letting the polynomial $p$ vary in a continuous way.
%\end{remark}

\section{Variables of polynomial rings}
In this section, we give some results on variables of polynomial rings. Most of them are classical or belong to the folklore. We include them for self-containedness and for lack of precise references.

\begin{definition}
Let $S$ be a ring and let $R\subset S$ be a subring. We denote by $\Aut_R(S)$ the group of automorphisms of the $R$-Algebra $S$. More precisely, 
\[\Aut_R(S)=\left\{\text{ automorphism of rings } f\colon S\to S \text{ such that } 
f|_R=\mathrm{id}_R\right\}.\]
\end{definition}

\begin{definition}
Let $R$ be a domain and $S$ be a polynomial ring in $n\ge 1$ variables over 
$R$, i.e.~$R\subset S$ and there exist $x_1,\dots,x_n\in S$ such that each element of $S$ can be written in a unique way as $f(x_1,\dots,x_n)$, where $f$ is
a polynomial in the $x_i$ with coefficients in $R$. An element $v\in S$ is called \emph{variable of the $R$-algebra $S$}
if there exists $f\in \Aut_R(S)$ such that $f(v)=x_1$.
\end{definition}

In the sequel, we often denote by $R[t]$ or $R[x]$ the polynomial ring in one variable over $R$, by $R[x,y]$ the polynomial ring in two variables over $R$
and by $R[x_1, \ldots, x_n]$ the polynomial ring in $n$ variables over $R$.

\begin{lemma}\label{Lem:CharacVariable}
Let $R$ be a domain, let $S=R[x_1,\dots,x_n]$ be the polynomial ring in $n$ variables over $R$ and let $v\in S$. The following conditions are equivalent:
\begin{enumerate}[label=$(\arabic*)$, ref=\arabic*]
\item
$v$ is a variable of the $R$-algebra $S$;
\item
The $R[t]$-algebra $S[t]/(v-t)$ is isomorphic to a polynomial ring in $n-1$ variables over $R[t]$.
\end{enumerate}
\end{lemma}
\begin{proof}
If $v$ is a variable, then 
there exists $f\in \Aut_R(S)$ such that $f(v)=x_1$. Using the natural inclusion $\Aut_R(S)\hookrightarrow \Aut_{R[t]}(S[t])$, we get isomorphisms of $R[t]$-algebras \[S[t]/(v-t)\iso S[t]/(x_1-t)\iso R[x_2,\dots,x_n,t]\iso R[t][x_2,\dots,x_n].\]

Conversely, suppose that the $R[t]$-algebra $S[t]/(v-t)$ is isomorphic to  a polynomial ring in $n-1$ variables over $R[t]$. This yields an $R[t]$-isomorphism $\psi\colon S[t]/(t-v)  \iso R[t][x_2,\dots,x_n]$. We then compose the isomorphisms of $R$-algebras
\[
	\begin{array}{cccccccc}
		S=R[x_1,\dots,x_n] & \iso & S[t]/(t-v) & \stackrel{\psi}{\longrightarrow} & 
		R[t][x_2,\dots,x_n] 
		%& \iso & R[x_1,\dots,x_n]
		\\
		f & \mapsto &f +(t-v)\cdot S[t]& & 
		%f(t,x_2,\dots,x_n) & \mapsto & f(x_1,x_2,\dots,x_n)	
	\end{array}
\]
and 
\[
	\begin{array}{ccc}
		R[t][x_2,\dots,x_n] & \iso & R[x_1,\dots,x_n]\\
		f(t,x_2,\dots,x_n)& \mapsto & f(x_1,x_2,\dots,x_n)
	\end{array}
\]
and obtain an element of $\Aut_R(S)$ that sends $v$ onto $x_1$. 
\end{proof}

\begin{lemma}\label{Lemm:ProdKxKxy}
Let $\k$ be a field, let $\k[x_1,\dots,x_n]$ be the polynomial ring in $n\ge 1$ variables over $\k$ and let $w\in \k[x_1,\dots,x_n]$ be a variable of this $\k$-algebra.

Then $\k[w]$ is factorially closed in $\k[x_1,\dots,x_n]$, i.e.~for all $f,g\in \k[x_1,\dots,x_n]\setminus \{0\}$, we have 
$fg\in \k[w]\Leftrightarrow f\in \k[w]\text{ and }g\in \k[w].$
\end{lemma}
\begin{proof}
If $f\in \k[w]$ and $g\in \k[w]$, then $fg\in \k[w]$, since $\k[w]$ is a subring of 
$\k[x_1,\dots,x_n]$.

Conversely, suppose that $fg\in \k[w]$. Choose $\psi\in \Aut_\k(\k[x_1,\dots,x_n])$ such that $\psi(w)=x_1$. Then, $\psi(f),\psi(g)\in \k[x_1,\dots,x_n]$ are two polynomials such that $\psi(f)\cdot \psi(g)\in \k[x_1]$. For each $i\ge 2$, the degree in $x_i$ satisfies $\deg_{x_i}(\psi(f))+\deg_{x_i}(\psi(f))=\deg_{x_i}(\psi(f)\cdot \psi(g))=0$, so $\deg_{x_i}(\psi(f))=\deg_{x_i}(\psi(f))=0$ since both elements are non-zero. Hence, $\psi(f),\psi(g)\in \k[x_1]$.
Applying $\psi^{-1}$, we get $f\in \k[w]$ and $g\in \k[w]$.
\end{proof}
\subsection{Variables of polynomial rings in two variables}
%Let us recall the following technical lemma:

We will need the following two technical lemmas:
\begin{lemma}\label{Lem:Dominant}
Let $\k$ be a field, let $w$ be a variable of the $\k$-algebra $\k[x,y]$ and let $v\in \k[x,y]$ be a polynomial. The following conditions are equivalent:
\begin{enumerate}[label=$(\arabic*)$, ref=\arabic*]
\item
$v\in \k[w]$;\label{vkw1}
\item
For each $u\in \k[w]$, the elements $u$ and $v$ are algebraically dependent over $\k$.\label{vkw2}
\item
There exists $u\in \k[w]\setminus \k$ such that $u$ and $v$ are algebraically dependent over $\k$.\label{vkw3}
\end{enumerate}
\end{lemma}
\begin{proof}The 
implications~$(\ref{vkw1})\Rightarrow (\ref{vkw2})\Rightarrow (\ref{vkw3})$ 
being clear, we only need to prove~$(\ref{vkw3})\Rightarrow (\ref{vkw1})$.
Replacing $v$ and $w$ with $f(v)$ and $f(w)$, for some $f\in\Aut_\k(\k[x,y])$, 
we can assume that $w=x$. Denoting by $\overline{\k}$ the algebraic closure of 
$\k$, we have $\overline{\k}[x]\cap \k[x,y]=\k[x]$, so we can assume that 
$\k=\overline{\k}$.

We then consider the morphism $\tau\colon \A^2_\k\to \A^2_\k$ given by $(x,y)\mapsto (u(x),v(x,y))$, which is dominant if and only if $u,v$ are algebraically independent over $\k$. It remains then to see that $\tau$ is dominant if $u\in \k[x]\setminus \k$ and $v\not\in \k[x]$. 
Let $v(x, y) = \sum_{i=0}^d v_i(x) y^i$, where $v_d \neq 0$ and $d > 0$.
For a general $a \in \k$, $u(x) = a$ 
has a solution $x_0$ such that $v_d(x_0) \neq 0$, 
since $\k$ is algebraically closed.
Hence $v(x_0, y)= b$ has a solution for all $b \in \k$. This proves that
$\tau$ is dominant.
\end{proof}
\begin{lemma}\label{Lem:TwoVariablesModulo}
Let $\k$ be a field, let $p\in \k[t]$ be an irreducible element and let $\k_p=\k[t]/(p)$ be the corresponding residue field.
Let $u,v\in \k[t][x,y]$ be elements such that $\k(t)[u,v]=\k(t)[x,y]$. Then, the classes $u_0,v_0\in \k_p[x,y]$ of $u,v$ satisfy one of the following properties, depending on the Jacobian determinant $\nu=\frac{\partial u}{ \partial x}\cdot \frac{\partial v}{\partial y}-\frac{\partial u}{ \partial y}\cdot \frac{\partial v}{\partial x}\in \k[t]\setminus \{0\}$:
\begin{enumerate}[label=$(\arabic*)$, ref=\arabic*]
\item\label{pdivnu}
If $p$ divides $\nu$, then $u_0,v_0$ are algebraically dependent over $\k_p$.
\item\label{pnotdivnu}
If $p$ does not divide $\nu$, then $\k_p[u_0,v_0]=\k_p[x,y]$. In particular, both $u_0$ and $v_0$ are variables of the $\k_p$-algebra $\k_p[x,y]$.
\end{enumerate}
\end{lemma}
\begin{proof}
Since $\k(t)[u,v]=\k(t)[x,y]$, there are polynomials $P,Q\in \k(t)[X,Y]$ such that $P(u,v)=x$, $Q(u,v)=y$. Moreover, the polynomial $\nu=\frac{\partial u}{ \partial x}\cdot \frac{\partial v}{\partial y}-\frac{\partial u}{ \partial y}\cdot \frac{\partial v}{\partial x}\in \k[t,x,y]$ belongs to $\k(t)^*$ and thus to $\k[t]\setminus \{0\}$. The element $\nu_0=\frac{\partial u_0}{ \partial x}\cdot \frac{\partial v_0}{\partial y}-\frac{\partial u_0}{ \partial y}\cdot \frac{\partial v_0}{\partial x}$ is then the class of $\nu$ in $\k_p$.

We write $P=\frac{\tilde{P}}{\alpha}$, $Q=\frac{\tilde{Q}}{\beta}$, where $\tilde{P},\tilde{Q}\in \k[t][X,Y]$, $\alpha,\beta\in \k[t]\setminus \{0\}$ and such that $p$ does not divide both $\alpha$ and $\tilde{P}$ (and the same for $\beta$ and $\tilde{Q}$). We then get 
\[
	\tilde{P}_0(u_0,v_0)=\alpha_0 x \, , \quad \tilde{Q}_0(u_0,v_0)=\beta_0 y
\]
where $\tilde{P}_0,\tilde{Q}_0\in \k_p[X,Y]$ are the classes of $\tilde{P},\tilde{Q}$ and $\alpha_0,\beta_0\in \k_p$ are the classes of $\alpha,\beta$. 

If $\alpha_0$ and $\beta_0$ are not equal to zero, then $\k_p[u_0,v_0]=\k_p[x,y]$. In particular, $u_0$ and $v_0$ are variables of the $\k_p$-algebra $\k_p[x,y]$ and $\nu_0\in \k_p^*$, so $p$ does not divide $\nu$.

If $\alpha_0=0$, then $\tilde{P}_0\not=0$ and $\tilde{P}_0(u_0,v_0)=0$ implies that $u_0$ and $v_0$ are algebraically dependent over $\k_p$. The same conclusion holds when $\beta_0=0$. In both cases, the Jacobian determinant $\nu_0$ is equal to zero, so $p$ divides $\nu$.

This yields~(\ref{pdivnu}) and~(\ref{pnotdivnu}).
\end{proof}

We recall the following classical result, essentially equivalent to the Jung-van der Kulk theorem:
\begin{lemma}\label{Lemm:JungExplicit}
Let $\k$ be a field, let $\k[x,y]$ be the polynomial ring in two variables over 
$\k$, let $f\in \Aut_\k(\k[x,y])$ and $u=f(x), v=f(y)\in \k[x,y]$.
%Using the degree of polynomials in $x,y$, the following holds: 
If $\deg(u)\ge \deg(v)>1$, then there exists a polynomial $P$ with coefficients in $\k$ such that $\deg(u-P(v))<\deg(u)$.
\end{lemma}
\begin{proof}
By van der Kulk's Theorem all automorphisms 
of $\k[x, y]$ are tame \cite{Jung1942,Kulk1953}. 
The statement is then a direct consequence of
\cite[Corollary~5.1.6]{Essen2000}.
\end{proof}

The following result is needed in the sequel. When the characteristic of $\k$ is zero, and $p=t$, it follows from \cite[Theorem~4]{FurterLength}. We adapt here the proof of \cite{FurterLength} for our purpose.

\begin{lemma}\label{Lemm:DegVariables}
Let $\k$ be a field, let $p\in \k[t]$ be an irreducible element and let $\k_p=\k[t]/(p)$ be the corresponding residue field.
If $v\in \k[t,x,y]$ is a variable of the $\k(t)$-algebra $\k(t)[x,y]$, then its class in $\k_p[x,y]$ is an element which belongs to $\k_p[w]\subset \k_p[x,y]$ for some variable $w$ of the $\k_p$-algebra $\k_p[x,y]$.
\end{lemma}

\begin{proof}
Let $f\in \Aut_{\k(t)}(\k(t)[x,y])$ such that $f(x)=v$, and let us define $u=f(y)$. We denote by $v_0\in \k_p[x,y]$ the class of $v$ and will use the degree of polynomials in $x,y$ with coefficients in $\k(t)$ or $\k_p$.

If $\deg(v)=1$, then $\deg(v_0)\le 1$. If $v_0\in \k_p$ the result follows by taking any variable for $w$, for instance $w=x$. Otherwise, $v_0=\alpha x+\beta y+\gamma$ for some $\alpha,\beta,\gamma\in \k_p$ with $(\alpha,\beta)\not=(0,0)$. This implies that $w=\alpha x+\beta y$ is a variable, as it is the component of an element of $\GL_2(\k_p)$, and the result follows.

We can thus assume that $\deg(v)>1$ and prove the result by induction on the pair $(\deg(v),\deg(u))$, ordered lexicographically.

$(i)$ If $\deg(u)\ge \deg(v)$, then there exists a polynomial $P\in \k(t)[X]$ such that $\deg(u-P(v))<\deg(u)$ (Lemma~\ref{Lemm:JungExplicit}). We can thus apply induction hypothesis to $(u-P(v),v)$, since $\k(t)[u,v]=\k(t)[u-P(v),v]$, 
and obtain the result.

$(ii)$ If $\deg(u)<\deg(v)$, we first replace $u$ with $u-\lambda$ for some $\lambda\in \k(t)$ and assume that $u\in \k(t)[x,y]$ is a polynomial in $x,y$ with no constant term. We then replace $u$ with $qu$ for some $q\in \k(t)^*$ and assume that $u\in \k[t][x,y]$ and the greatest common divisor in $\k[t]$ of the coefficients of $u$ (as a polynomial in $x,y$) is equal to $1$. One can then define the class $u_0\in \k_p[x,y]$ of $u$, which is not equal to zero. Moreover, $u_0$ does not belong to $\k_p$, since $u_0$ has no constant term.

If $v_0$ is a variable of the $\k_p$-algebra $\k_p[x,y]$, then we are done. Otherwise, $u_0,v_0$ are algebraically dependent over $\k_p$ (Lemma~\ref{Lem:TwoVariablesModulo}).

Since the pair $(\deg(v),\deg(u))$ is smaller than $(\deg(u),\deg(v))$, we can apply induction hypothesis and get a variable $w\in \k_p[x,y]$ such that $u_0\in  \k_p[w]$. The fact that $u_0$ and $v_0$ are algebraically dependent over $\k_p$ and that $u_0\not\in \k_p$ imply that $v_0\in \k_p[w]$ (Lemma~\ref{Lem:Dominant}).
\end{proof}

We finish this section with several results relating variables and $\A^1$-bundles.

\begin{lemma}\label{Lem:PolynomTwoVariables}
Let $\k$ be a field and let $P\in \k[x,y]$. Then, the following conditions are equivalent:
\begin{enumerate}[label=$(\arabic*)$, ref=\arabic*] 
\item\label{Pvkxy}
The polynomial $P$ is a variable of the $\k$-algebra $\k[x,y]$.
\item\label{Pktxy}
The $\k[t]$-algebra $\k[t,x,y]/(P-t)$ is a polynomial ring in one variable over $\k[t]$.
\item\label{Pktxyfield}
The $\k(t)$-algebra $\k(t)[x,y]/(P-t)$ is a polynomial ring in one variable over $\k(t)$.
\item\label{Ptrivial}
The morphism $\A^2_\k\to \A^1_\k$ given by $(x,y)\mapsto P(x,y)$ is a trivial $\A^1$-bundle.
\item\label{Ploctrivial}
The morphism $\A^2_\k\to \A^1_\k$ given by $(x,y)\mapsto P(x,y)$ is a trivial $\A^1$-bundle over some dense open subset $U\subset \A^1_\k$.
\end{enumerate}
\end{lemma}
\begin{proof}
$(\ref{Pvkxy})\Leftrightarrow(\ref{Pktxy})$ follows from Lemma~\ref{Lem:CharacVariable}.

$(\ref{Pvkxy})\Leftrightarrow(\ref{Ptrivial})$: By definition, $(\ref{Pvkxy})$ is equivalent to the existence of $f\in \Aut_\k(\k[x,y])$ such that $f(x)=P$. As $f=\varphi^*$ for some $\varphi\in \Aut(\A^2_\k)$, this is equivalent to ask for $\varphi\in \Aut(\A^2_\k)$ that $\pr_x \circ \varphi$ is the map $(x,y)\mapsto P(x,y)$, where $\pr_x\colon \A^2_\k\to \A^1_\k$ is given by $(x,y)\mapsto x$. This yields the equivalence $(\ref{Pvkxy})\Leftrightarrow(\ref{Ptrivial})$.

$(\ref{Pktxy})\Rightarrow(\ref{Pktxyfield})$ is trivially true.

$(\ref{Pktxyfield})\Rightarrow (\ref{Ploctrivial})$: Assertion $(\ref{Pktxyfield})$ corresponds to say that the generic fibre of $(x,y)\mapsto P(x,y)$ is isomorphic to $\A^1_{\k(t)}$. This yields $(\ref{Ploctrivial})$.

$(\ref{Ploctrivial})\Rightarrow (\ref{Pvkxy}):$ 
Assume that the subset $U$ given in $(\ref{Ploctrivial})$
contains a $\k$-rational point.
Replacing $P$ with $P+\lambda$, $\lambda\in \k$, 
one can assume that $0$ belongs to the open subset 
$U$. One then observes that the curve 
$\Gamma\subset \A^2_\k$ given by $P=0$ is isomorphic to $\A^1_\k$ 
and equivalent to a line by a birational map of $\A^2_\k$, hence 
can be contracted by a birational map of $\A^2_\k$. 
By \cite[Proposition 2.29]{BlancFurterHemmig}, there exists 
an automorphism $\varphi\in \Aut(\A^2_\k)$ which sends $\Gamma$ 
onto the line given by $x=0$. This implies that $P$ is a variable of 
the $\k$-algebra $\k[x,y]$.

If $U$ contains no $\k$-rational point, then $\k$
is a finite field and thus it is perfect. For a finite Galois extension
$\k' \supset \k$ the subset $U$ contains a $\k'$-rational point. By
the argument above, $P$ is a variable of the $\k'$-algebra $\k'[x, y]$
and hence $\k'[x, y] = \k'[P, Q]$ for some $Q \in \k'[x, y]$.
Since $P$ is a polynomial with coefficients in $\k$,
it is fixed under
the action of the Galois group $G = \Gal(\k' / \k)$ on $\k[x, y] = \k'[P, Q]$.
For each $\sigma \in G$, there exists $(a_\sigma, b_\sigma) \in (\k')^\ast \ltimes \k'[T]$ with
\[
	\sigma(Q) = a_\sigma Q + b_\sigma(P).
\]
We can then find $d\ge 0$ such that  $\{b_\sigma\mid \sigma\in G\}$ is contained in the finite dimensional $\k'$-vector subspace $V_d=\{f\in \k'[T]\mid \deg(P)\le d\}\subset \k'[T]$. 
Thus $\sigma \mapsto (a_\sigma, b_\sigma)$ defines
an element of $H^1(G, (\k')^\ast \ltimes V_d)$.
As $H^1(G, (\k')^\ast )=\{1\}$ and $H^1(G,V_d)=\{1\}$ \cite[Proposition~1, 2, Chp. X]{Serre68}, we have $H^1(G, (\k')^\ast \ltimes V_d)=\{1\}$. The fact that $(a_\sigma, b_\sigma)$ is a trivial cocycle corresponds to the existence of
a polynomial $Q_0 \in \k[x, y]$ such that $\k[P, Q_0] = \k[x, y]$. 
This implies that $P$ is a variable of the $\k$-algebra $\k[x, y]$.
\end{proof}

We recall the following classical result:
\begin{lemma}\label{Lem:BCWdim1}
	Let $\k$ be a field, let $Z$ be an  affine variety over $\k$,
	all of its irreducible components being surfaces,
	let $U \subseteq \A^1_{\k}$ be a dense open subset and let 
	$\pi \colon Z \to U$ be a dominant morphism. 
	Then, the following conditions are equivalent:
	\begin{enumerate}[label=$(\arabic*)$, ref=\arabic*]
		\item\label{PPPtrivialA1bundle}
			The morphism $\pi\colon Z\to U$ is a trivial $\A^1$-bundle.
		\item\label{PPPloctrivialA1bundle}
			The morphism $\pi\colon Z\to U$ is a locally 
			trivial $\A^1$-bundle.
		\item\label{PPPeachfibA1}
			For each maximal ideal $\mathfrak{m} \subset \k[U]$, 
			the fibre $\pi^{-1}(\mathfrak{m})\subset Z$ is isomorphic 
			to $\A^1_{\kappa(\mathfrak{m})}$ and the generic fibre
			of $\pi$ is isomorphic to $\A^1_{\k(t)}$.
	\end{enumerate}
\end{lemma}

\begin{proof}
	The implications $\eqref{PPPtrivialA1bundle} \Rightarrow
	\eqref{PPPloctrivialA1bundle} \Rightarrow \eqref{PPPeachfibA1}$
	are obvious. 
	Assume $\eqref{PPPeachfibA1}$ holds.
	Since each irreducible component of $Z$ has 
	dimension two, it follows that each of these irreducible components
	is mapped dominantly onto $U$ via $\pi$. 
	Thus $\pi$ is flat. By \cite[Corollary~3.2]{Asanuma87}
	it follows now from~$\eqref{PPPeachfibA1}$ 
	that $\k[Z]_{\mathfrak{m}}$ is a polynomial ring
	in one variable over $\k[U]_{\mathfrak{m}}$ for all
	maximal ideals $\mathfrak{m} \subset \k[U]$. Hence, by 
	\cite{BCW}, the morphism $\pi$ is a vector bundle 
	with respect to the Zariski topology and since
 	$\k[U]$ is a principal ideal domain, $\pi$
	is a trivial $\A^1$-bundle.
\end{proof}

\begin{lemma}\label{Lemm:ZinA3U}
Let $\k$ be a field, $P\in \k[t,x,y]$ be a polynomial which is a variable of the 
$\k(t)$-algebra $\k(t)[x,y]$, let $U\subset \A^1_\k=\Spec(\k[t])$ be a 
dense open subset, let $Z\subset U\times\A^2_\k=\Spec(\k[U][x,y])$  be the 
hypersurface given by $P = 0$ 
and let $\pi\colon Z\to U$ be the morphism $(t,x,y)\mapsto t$. Then, 
the following conditions are equivalent:
\begin{enumerate}[label=$(\roman*)$, ref=\roman*]
\item\label{PPVarktU}
$P$ is a variable of the $\k[U]$-algebra $\k[U][x,y]$;
\item\label{PPisoA2fibrU}
There is an isomorphism $\varphi\colon  U\times \A^1_\k\iso Z$ such that $\pi\varphi$ is the projection $(t,x)\mapsto t$;
\item\label{PPtrivialA1bundleU}
The morphism $\pi\colon Z\to U$ is a trivial $\A^1$-bundle;
\item\label{PPloctrivialA1bundleU}
The morphism $\pi\colon Z\to U$ is a locally trivial $\A^1$-bundle;
\item\label{PPeachfibA1U}
For each maximal ideal $\mathfrak{m} \subset \k[U]$, the fibre $\pi^{-1}(\mathfrak{m})\subset Z$ is isomorphic to $\A^1_{\kappa(\mathfrak{m})}$.
\end{enumerate}
\end{lemma}

\begin{proof}
$(\ref{PPVarktU})\Rightarrow (\ref{PPisoA2fibrU})$:
If $P$ is a variable of the $\k[U]$-algebra $\k[U][x,y]$, there exists 
$f\in \Aut_{\k[U]}(\k[U][x,y])$ such that $f(x)=P$. The element $f$ is then equal to $\psi^*$ for some $\psi\in \Aut(U\times \A^2_\k)$ such that $\pi \psi=\pi$. Hence, $\psi(Z)$ is the closed subset of $U\times \A^2_\k$ given by $x=0$. 
Let $\theta\colon U\times\A^2_\k\to U\times\A^1_\k$ 
be the projection given by $(t,x,y)\mapsto (t,y)$. 
The composition $\theta\circ \psi$ restricts to an
isomorphism $Z \iso U\times\A^1_\k$ 
that we denote by $\varphi^{-1}$. Thus $\mathrm{pr}_1\circ \varphi^{-1}=\pi$, where $\mathrm{pr}_1\colon U\times\A^1_\k\to U$ is the projection on the first factor. 
This yields $(\ref{PPisoA2fibrU})$.

$(\ref{PPisoA2fibrU})\Leftrightarrow (\ref{PPtrivialA1bundleU})$: is the definition of a trivial $\A^1$-bundle.

$(\ref{PPtrivialA1bundleU})\Rightarrow (\ref{PPloctrivialA1bundleU})\Rightarrow (\ref{PPeachfibA1U})$ are obvious.

$(\ref{PPeachfibA1U})\Rightarrow (\ref{PPVarktU})$: 
This follows 
from the implication $\eqref{PPPeachfibA1} \Rightarrow \eqref{PPPtrivialA1bundle}$
of Lemma~\ref{Lem:BCWdim1}.

\end{proof}

\begin{corollary}\label{Cor:ZinA3}
Let $\k$ be a field, let $P\in \k[t,x,y]$ be a polynomial which is 
a variable of the $\k(t)$-algebra $\k(t)[x,y]$, let 
$Z\subset \A^3_\k=\Spec(\k[t,x,y])$  be the hypersurface given by $P = 0$ 
and let $\pi\colon Z\to \A^1_\k$ be the morphism $(t,x,y)\mapsto t$. Then, the following conditions are equivalent:
\begin{enumerate}[label=$(\roman*)$, ref=\roman*]
\item\label{PPVarkt}
$P$ is a variable of the $\k[t]$-algebra $\k[t][x,y]$;
\item\label{PPisoA2fibr}
There is an isomorphism $\varphi\colon \A^2_\k\iso Z$ such that $\pi\varphi$ is the projection $(t,x)\mapsto t$;
\item\label{PPisoA2}
There is an isomorphism $\varphi\colon \A^2_\k\iso Z$;
\item\label{PPtrivialA1bundle}
The morphism $\pi\colon Z\to \A^1_\k$ is a trivial $\A^1$-bundle;
\item\label{PPloctrivialA1bundle}
The morphism $\pi\colon Z\to \A^1_\k$ is a locally trivial $\A^1$-bundle;
\item\label{PPeachfibA1}
For each maximal ideal $\mathfrak{m} \subset \k[t]$, the fibre $\pi^{-1}(\mathfrak{m})\subset Z$ is isomorphic to $\A^1_{\kappa(\mathfrak{m})}$.
\end{enumerate}
\end{corollary}

\begin{proof}Applying Lemma~\ref{Lemm:ZinA3U} with $U=\A^1_\k$, we obtain the equivalence between $(\ref{PPVarkt})$-$(\ref{PPisoA2fibr})$-$(\ref{PPtrivialA1bundle})$-$(\ref{PPloctrivialA1bundle})$-$(\ref{PPeachfibA1})$.

We then observe that $(\ref{PPisoA2fibr})$ implies $(\ref{PPisoA2})$. It remains then to prove $(\ref{PPisoA2})\Rightarrow (\ref{PPtrivialA1bundle})$. As $P$ is a variable of the  $\k(t)$-algebra $\k(t)[x,y]$, the generic fibre of $\pi\colon Z\to \A^1_\k$ is isomorphic to $\A^1_{\k(t)}$, so $\pi$ is a trivial $\A^1$-bundle over some dense open subset $U\subset \A^1_\k$. The fact that $Z$ is isomorphic to $\A^2_\k$ implies then that $\pi$ is a trivial $\A^1_\k$-bundle. (Implication $(\ref{Ploctrivial})\Rightarrow(\ref{Ptrivial})$ of Lemma~\ref{Lem:PolynomTwoVariables}).
\end{proof}

\subsection{Non-trivial embeddings in positive characteristic}
In this paragraph, we recall the existence of non-trivial embeddings in positive characteristic. The family of examples that we give below seems classical (the case $\A^1_\k\hookrightarrow \A^2_\k$ with parameters equal to $1$ corresponds in particular to~\cite[Exercise 5(iii) in $\S 5$]{Essen2000}). We give the (simple) proof here for a lack of a precise reference and for self-containedness.
\begin{lemma}\label{Lemm:EmbCharP}
For each field $\k$ of characteristic $p>0$, each $a \in \k$, $b \in \k^\ast$
and each integer $q\ge 0$, the morphism
\[\begin{array}{rrcl}
\rho\colon &\A^1_\k& \hookrightarrow & \A^2_\k\\
&u&\mapsto&  (u^{p^2},\frac{1}{b}(a u^{pq}+ u))
\end{array}\]
is a closed embedding, with image being the closed curve of 
$\A^2_\k=\Spec(\k[x,y])$ given by
\[
	x+ a^{p^2}x^{pq}-b^{p^2} y^{p^2}=0.
\]
\end{lemma}
\begin{proof}We first compute the equality
\[
	b^{p}
	\left(\frac{1}{b}(a u^{pq}+ u)\right)^p-a^p(u^{p^2})^q=
	(a^pu^{p^2q}+u^p)-a^p u^{p^2q}=u^p,
\]  
which shows that $\rho(\A^1_\k)$ is contained in the closed curve $\Gamma\subset \A^2_\k$ given by $P=0$, where 
$P= x - (b^{p}y^{p}-a^px^{q})^p = x-b^{p^2}y^{p^2}+a^{p^2}x^{qp}\in \k[x,y]$. The equality also yields $u^p\in \k[u^{p^2}, \frac{1}{b}(a u^{pq}+u)]$ 
and thus yields 
$\k[u^{p^2}, \frac{1}{b}(a u^{pq}+u)]=\k[u]$, which implies that $\rho$ is a closed embedding. 
It remains to see that the degree of $\rho$ (maximum of the degree of both components) is equal to the degree of $P$, to obtain that $P$ is irreducible and that it defines the irreducible curve $\rho(\A^1_\k)$. For $a = 0$, this follows, since $\deg(\rho)=p^2=\deg(P)$. For $a\not=0$, we have $\deg(\rho)=\max(p^2,pq)=\deg (P)$.
\end{proof}

To show that the above embeddings are not equivalent to the standard one, when $q\ge 2$ is not a multiple of $p$ and $a,b \not=0$, 
one could make argument on the degree of the components 
(no one divides the other) or can use the characterisation of 
variables given in Lemma~\ref{Lem:CharacVariable}
to show that 
$P=x+a^{p^2}x^{pq}-b^{p^2}y^{p^2} \in \k[x,y]$ is not a variable, by proving that $\k[x,y,t]/(P-t)$ is not a polynomial ring in one variable over $\k[t]$, as we do in 
Lemma~\ref{Lemm:FormofA1} below. 
The second way has the advantage to give examples in any dimension (see Proposition~\ref{Prop:EmbeddingsCharP}).
This is related to the forms of the affine line over non-perfect fields (for more details on this subject, see \cite{Russell1970}). 

\begin{lemma}\label{Lemm:FormofA1}
For each field $\k$ of characteristic $p>0$, each $b \in \k^\ast$ and 
each integer $q\ge 2$, not a multiple of $p$, the curve 
\[
	\Gamma=\Spec \left(\k(t)[x,y]/(x+a^{p^2} x^{pq}- b^{p^2} y^{p^2}-t)\right)
\]
is not isomorphic to $\A^1_{\k(t)}$, but after extension of scalars to $\k(t^{1/p})$ we have an isomorphism \[\Gamma_{\k(t^{1/p})}\iso \A^1_{\k(t^{1/p})}.\]
\end{lemma}
\begin{proof}
After extending the scalars to $\k(t^{1/p^2})$, the curve $\Gamma$ becomes 
\[
	\begin{array}{rcl}
		\Gamma_{\k(t^{1/p^2})}&=&\Spec \left(
			\k(t^{1/p^2})[x,y]/(x+a^{p^2}x^{pq} - b^{p^2}y^{p^2}-t) \right) 
			\vspace{0.1cm}\\

		&=&\Spec 
		\left(\k(t^{1/p^2})[x,y]/
		\left(x+ a^{p^2}x^{pq}-b^{p^2} \left(y+\frac{t^{1/p^2}}{b} \right)^{p^2} 
		\right) \right) \, . 	\end{array}
\]
Replacing $y$ with $y+\frac{t^{1/p^2}}{b}$ and applying 
Lemma~\ref{Lemm:EmbCharP} we obtain an isomorphism
\[
	\begin{array}{rcl}
		\A^1_{\k(t^{1/p^2})}& \iso & \Gamma_{\k(t^{1/p^2})}\\
		u & \mapsto & \left(u^{p^2}, \frac{1}{b}
		\left(au^{pq}+u-t^{1/p^2}\right)\right) \, . %\\
		%y+t^{1/p^2}-(y^p+t^{1/p}-x^q)^q&\mapsfrom & (x,y).
	\end{array}
\]
Replacing then $u$ with $u+t^{1/p^2}$ we get an isomorphism defined over 
$\k(t^{1/p})$:
\[
	\begin{array}{rrcl}
		\nu\colon& \A^1_{\k(t^{1/p})}& \iso & \Gamma_{\k(t^{1/p})}\\
		& u & \mapsto & \left(u^{p^2}+t,
		\frac{1}{b} \left(a \left(u^p+t^{1/p}\right)^q+u\right)\right) \, .%\\
		%y-(y^p+t^{1/p}-x^q)^q&\mapsfrom & (x,y)
	\end{array}
\]
It remains that no isomorphism $\hat\nu\colon \A^1_{\k(t)}\iso \Gamma_{\k(t)}$ exists. If $\hat\nu$ exists, then $\nu^{-1}\hat\nu\in \Aut(\A^1_{\k(t^{1/p})})$ would be given by $u\mapsto \alpha u+\beta$, with $\alpha,\beta\in \k(t^{1/p})$, $\alpha\not=0$.  The second coordinate of $\hat\nu(u)$ would then be
\[
	\frac{1}{b}
	\left(a\left((\alpha u+\beta)^p+t^{1/p}\right)^{q}+(\alpha u+\beta) \right) 
	\in \k(t)[u] \, .
\]
The coefficient of $u$ being $\frac{\alpha}{b}$, we get $\alpha\in \k(t)$. Remembering that $q\ge 2$, the coefficient of $u^{p(q-1)}$ is equal to 
$\frac{a}{b} q \alpha^{p(q-1)}(\beta^p+t^{1/p})$. As $\beta^p\in \k(t)$ we have $\beta^p+t^{1/p}\notin \k(t)$. Impossible, since $q$ is not a multiple of $p$ and $\alpha\not=0$.  
\end{proof}

\begin{corollary}\label{Cor:NotVariablesCharP}
For each field $\k$ of characteristic $p>0$, each integer $q\ge 2$ which is not a multiple of $p$, each $\lambda,\mu\in \k^*$ and each integer $n\ge 2$, the polynomial 
\[f=x_1+\lambda x_1^{pq}+\mu x_2^{p^2}\in \k[x_1,x_2]\subset\k[x_1,\dots,x_n]\] is not a variable of the $\k$-algebra $\k[x_1,\dots,x_n]$.
\end{corollary}

\begin{proof}
Showing that $f$ is not a variable of $\k[x_1,\dots,x_n]$ is equivalent to ask that the $\k[t]$-algebra $\k[x_1,x_2,\dots,x_n,t]/(f-t)$ is not a polynomial ring in $n-1$ variables over $\k[t]$ (Lemma~\ref{Lem:CharacVariable}). It suffices then to show that $A_n=\k(t)[x_1,\dots,x_n]/(f-t)$ is not a polynomial ring in $n-1$ variables over $\k(t)$.

We first prove the result for $n=2$. By extending the scalars, we can assume that $\lambda=a^{p^2}$ and $\mu=-b^{p^2}$ for some $a,b \in \k^*$.  Lemma~\ref{Lemm:FormofA1} then shows that $A_2=\k(t)[x_1,x_2]/(f-t)$ is not a polynomial ring in one variable.

As $A_n=A_2[x_3,\dots,x_n]$, the positive answer to the cancellation problem for curves \cite{AHE} implies that $A_n$ is not a polynomial ring in $n-1$ variables over $\k(t)$ for each $n\ge 2$.
\end{proof}

\begin{proposition}\label{Prop:EmbeddingsCharP}
For each field $\k$ of characteristic $p>0$, each integer $q\ge 2$ which is not a multiple of $p$, each $a \in \k^*$ and each $n\ge 1$, the morphism
\[\begin{array}{rrcl}
\rho\colon &\A^n& \hookrightarrow & \A^{n+1}\\
&(x_1,\dots,x_n)&\mapsto&  (x_1^{p^2}, a x_1^{pq} + x_1,x_2,\dots,x_n)
\end{array}\]
is a closed embedding, which is not equivalent to the standard one.
\end{proposition}
\begin{proof}
It follows from Lemma~\ref{Lemm:EmbCharP} that $\rho$ is a closed embedding and that its image is given by the hypersurface with equation 
$f=0$, where \[f=x_1+ a^{p^2}x_1^{pq}-x_2^{p^2}\in \k[x_1,x_2]\subset
\k[x_1,\dots,x_n].\]

It remains to show that $f$ is not a variable of $\k[x_1,\dots,x_n]$, 
which follows from Corollary~\ref{Cor:NotVariablesCharP}.
\end{proof}
\section{Liftings of automorphisms and the proof of Theorem~\ref{Thm:StdEmbSL2}}
\subsection{Lifting of automorphisms of $\A^3_\k$ to affine modifications}
\begin{lemma}
\label{lem.Key}
Let $\k$ be a field and let
\[R = \k[t,u,x,y] / (t^nu-h(t,x,y))\]
where $n\ge 1$ and $h \in \k[t,x, y]$ is a polynomial  such that 
$h_0=h(0,x,y)\in \k[x,y]$ does not belong to $\k[w]$ for each variable $w\in \k[x,y].$

\begin{enumerate}[label=$(\arabic*)$, ref=\arabic*]
\item\label{DescriptionR}
Every element of $R\setminus \k[t,x,y]$ can be written as
\[s+\sum_{i=1}^{m} f_i u^i\] 
where $s\in \k[t,x,y]$, $m\ge 1$, $f_1,\dots,f_m\in \k[t,x,y]$ are polynomials of degree $<n$ in $t$, and $f_m\not=0$.
\item\label{LimitAfterMult}
If $f\in R\setminus \k[t,x,y]$ is written as in $(\ref{DescriptionR})$ and $d=\nu(f_m)$ is the valuation of $f_m$ in $t$, then $0\le d<n$ and $t^{mn-d}f=g(t,x,y)\in \k[t,x,y]$ satisfies $g(0,x,y)\in h_0\cdot \k[x,y]\setminus \{0\}$.
\item\label{IdealRktxy}
Writing $I\subset \k[t,x,y]$ for the ideal $(t^n,h)$, we have $t^nR\cap \k[t,x,y]=I$.
\item
\label{Preservesktxy}
Every element of $\Aut_{\k[t]}(R)$ preserves the sets $\k[t,x,y]$ and $I$.
\end{enumerate}
\end{lemma}

\begin{proof}
(\ref{DescriptionR}): We first prove that every element of $R\setminus \k[t,x,y]$ has the desired form. Every element of $R$ can be written as $\sum_{i=0}^m f_iu^i$ for some polynomials $f_i\in \k[t,x,y]$. We denote by $r$ the biggest integer such that $\deg_t(f_r)\ge n$. If $r=0$ or if no such integer exists, we are done. Otherwise, we write $f_r=t^nA+B$ for some $A,B\in \k[t,x,y]$ with $\deg_t(B)<n$. Then, replacing $f_{r-1}u^{r-1}+f_ru^r=f_{r-1}u^{r-1}+(t^nA+B)u^r$ in $\sum_{i=0}^m f_iu^i$ 
with $(f_{r-1}+h(t,x,y)A)u^{r-1}+Bu^r$ decreases the integer $r$. After finitely many such steps, we obtain the desired form.

(\ref{LimitAfterMult}): We write $f=s+\sum_{i=1}^{m} f_i u^i=s+\sum_{i=1}^{m} f_i \frac{h^i}{t^{ni}}$ as in (\ref{DescriptionR}), we write $d=\nu(f_m)$, which satisfies $0\le d<n$ (since $f_m\not=0$ and $\deg_t(f_m)<n$), and obtain $g=t^{mn-d}f=st^{mn-d}+\sum_{i=1}^{m} f_i h^it^{mn-d-ni}$. In particular, $g\in \k[t,x,y]$ and it
satisfies $g(0,x,y)=(h_0)^m\cdot r$, where $r\in \k[x,y]$ is obtained by replacing $t=0$ in $\frac{f_m}{t^d}\in \k[t,x,y]$. From $\{r,h_0\}\subset \k[x,y]\setminus \{0\}$, we deduce $g(0,x,y)\not=0$.

(\ref{IdealRktxy}): The inclusion $I\subset t^nR$ follows from $\{t^n,h\}=\{t^n\cdot 1,t^n\cdot u\}\subset t^n R$. To show that $t^nR \cap \k[t, x, y] \subset I$, we take $f\in R$ such that $t^nf\in \k[t,x,y]$ and show that $t^nf\in I$. If $f\in \k[t,x,y]$, then 
$t^nf\in t^n \k[t, x, y]\subset I$. Otherwise, we write 
$f=s+\sum_{i=1}^{m} f_i u^i$ as in (\ref{DescriptionR}) and use (\ref{LimitAfterMult}) to obtain that $g=t^{mn-d}f \in \k[t, x, y]$ with $0\le d=\nu(f_m)<n$, and we 
get $g(0,x,y) \neq 0$. The fact that $t^nf\in \k[t, x, y]$ implies then that $n > mn-d$, whence $n>d > (m-1)n$, so $m=1$. Hence 
$t^nf=t^n(s+f_1u)=t^ns + h f_1 \in I$.

 (\ref{Preservesktxy}): Using (\ref{IdealRktxy}) it suffices to show that every $\psi\in \Aut_{\k[t]}(R)$ preserves $\k[t,x,y]$. The algebra $R$ is canonically isomorphic to 
$\k[t,x,y][\frac{h}{t^n}] \subset  \k(t)[x,y]$.
 Since $\k(t)[x,y]$ is the localisation of $\k[t,x,y][\frac{h}{t^n}]$ in the multiplicative
 system $\k[t] \setminus \{ 0 \}$, we get a natural inclusion
 $\Aut_{\k[t]}(R)\subset\Aut_{K}K[x,y]$, with $K=\k(t)$. 

Suppose for contradiction that some $\psi\in \Aut_{\k[t]}(R)$ satisfies $\psi(\k[t,x,y])\not\subset \k[t,x,y]$. This implies that $\psi(x)\notin \k[t,x,y]$ or $\psi(y)\notin \k[t,x,y]$. 
  We assume that $\psi(x)\notin \k[t,x,y]$ (the case $\psi(y)\notin \k[t,x,y]$ being similar) and use (\ref{LimitAfterMult}) to obtain an integer $l>0$ such that $g=t^{l}\psi(x)\in \k[t,x,y]$ satisfies $g(0,x,y)\in h_0\cdot \k[x,y]\setminus \{0\}$. Since $\psi\in\Aut_{\k[t]}(R)\subset\Aut_{K}K[x,y]$, the element
 $\psi(x)$ is a variable of $K[x,y]$ and the same holds for $g(t,x,y)=t^{l}\psi(x)$. By Lemma~\ref{Lemm:DegVariables}, $g(0,x,y)$ belongs to $\k[w]$ for some variable $w\in \k[x,y]$. The fact that $g(0,x,y)\in h_0\cdot \k[x,y]\setminus \{0\}$ implies then that $h_0\in \k[w]$ (Lemma~\ref{Lemm:ProdKxKxy}), contradicting the hypothesis.
\end{proof}
\begin{corollary}\label{Cor:ExtensionAutomorphisms}
Let $\k$ be a field and
\[R = \k[t,u,x,y] / (t^nu-h(t,x,y))\]
where $n\ge 1$ and $h \in \k[t,x, y]$ is a polynomial such that 
$h_0=h(0,x,y)\in \k[x,y]$ does not belong to $\k[w]$ for each variable $w\in \k[x,y].$
Writing $I$ the ideal $(t^n,h)\subset \k[t,x,y]$, we obtain a group isomorphism
 \[\begin{array}{rcl}
 \Aut_{\k[t]}(R)&\iso & \{\psi\in \Aut_{\k[t]}(\k[t,x,y])\mid \psi(I)=I\}\\
 \varphi & \mapsto & \varphi|_{\k[t,x,y]}\end{array}\]
\end{corollary}
\begin{proof}
According to Lemma~\ref{lem.Key}(\ref{Preservesktxy}), every element  $\varphi\in\Aut_{\k[t]}(R)$ preserves $\k[t,x,y]$ and~$I$, and thus restricts to an element $\psi\in \Aut_{\k[t]}(\k[t,x,y])$ that preserves $I$.

Conversely, each automorphism $\psi\in \Aut_{\k[t]}(\k[t,x,y])$ that preserves $I$ induces an automorphism of $R=\k[t,x,y][\frac{I}{t^n}]=\k[t,x,y][\frac{h}{t^n}]$. This latter is uniquely determined by $\psi$, since the morphism $\Spec(R)\to \Spec(\k[t,x,y])$ given by the inclusion $\k[t,x,y]\hookrightarrow R$ is birational.
\end{proof}
\begin{remark}
According to \cite[Definition~1.1 and Proposition~1.1]{KalimanZaidenberg99}, $\Spec(R)$ is the affine modification of $\A^3_\k=\Spec(\k[t,x,y])$ with locus $(I, t^n)$. It is thus natural that every automorphism of $\A^3_\k$ fixing the ideal lift to an automorphism of $\Spec(R)$. The interesting part of Corollary~\ref{Cor:ExtensionAutomorphisms} consists then in saying that all automorphisms
of the $\k[t]$-algebra $R$ are of this form.
\end{remark}
\subsection{Application of liftings to the case of $\SL_2$}
We will apply Corollary~\ref{Cor:ExtensionAutomorphisms} to the variety $\SL_2\subset \A^4$ given by 
\[\SL_2 = \left\{\left. \begin{pmatrix}
				  x & t \\
				  u & y
			     \end{pmatrix}\in \A^4 \right| xy - tu = 1 \right\},\]
			     and obtain Proposition~\ref{AutoSL2fixingX} below.
Before we give a proof, let us recall the following basic facts on
the coordinate ring of the variety $\SL_2$.
\begin{lemma}
	\label{lem:R_is_UFD}
	Let $R$ be the coordinate ring of $\SL_2$, i.e.~$R = \k[t,u,x,y] / (xy-tu-1)$.
	Then $R$ is a unique factorisation domain and the units of $R$
	satisfy $R^\ast = \k^\ast$.
\end{lemma}

\begin{proof}
	Since the localisation $R_t = \k[t, \frac{1}{t}] [x, y]$ is a unique 
	factorisation domain, we only have to see that 
	$tR$ is a prime ideal of $R$, by~\cite[Theorem~20.2]{Matsumura89}.  
	This is the case, since $R/tR \simeq \k[u, x,y]/(xy-1)$
	is an integral domain. Moreover, we have 
	$R^\ast \subseteq (R_t)^\ast =\{\mu t^n\mid \mu\in\k^*,n\in\mathbb{Z}\}$.
	Since $t^n$ is invertible in $R$ if and only if $n = 0$, it follows 
	that $R^\ast = \k^\ast$.
\end{proof}

\begin{proposition}
\label{AutoSL2fixingX}
We consider the morphisms
\[\begin{array}{ccccc}
\SL_2=\Spec(\k[t,u,x,y] / (xy-tu-1)) & \stackrel{\eta}{\to} & \A^3_\k & \stackrel{\pi }{\to}& \A^1_\k\\
(t,u,x,y)& \mapsto & (t,x,y) & \mapsto & t\end{array}\]
and denote by $X\subset \SL_2$ the hypersurface given by $t=1$ and by $\Gamma\subset \A^3_\k$ the closed curve given by $t=xy-1=0$. 

Then, the birational morphism $\eta\colon \SL_2\to \A^3_{\k}$ 
yields  a group isomorphism
\[\begin{array}{rcl}
\{g\in \Aut(\SL_2)\mid g(X)=X\}& \iso & \{g\in \Aut(\A^3_\k)\mid \pi g=\pi, g(\Gamma)=\Gamma\}\\
 g & \mapsto & \eta g \eta^{-1}.
\end{array}\]
We moreover have 
\[
	\begin{array}{rcl}
		\{g\in \Aut(\SL_2)\mid g(X)=X\}&=&\{g\in \Aut(\SL_2)\mid 
		\pi\eta g=\pi\eta\},\\
		&=&\{g\in \Aut(\SL_2)\mid g^*(t)=t\}.
	\end{array}
\]

\end{proposition}

\begin{proof}
Every automorphism $g$ of $\A^3_\k=\Spec(\k[t,x,y])$ yields an automorphism
$g^*\in \Aut_\k(\k[t,x,y])$. Moreover, the condition $\pi g=g$ corresponds to 
$g^*(t)=t$, and the condition $g(\Gamma)=\Gamma$ to $g^*(I)=I$, where $I\subset \k[t,x,y]$ is the ideal of $\Gamma$, generated by $t$ and $xy-1$. The isomorphism $\Aut(\A^3_\k)\to \Aut_\k(\k[t,x,y])$ then yields a bijection
\[
	\begin{array}{rcl}
		\{g\in \Aut(\A^3_\k)\mid \pi g=\pi, g(\Gamma)=\Gamma\} & \iso & 
		\{\psi\in \Aut_{\k[t]}(\k[t,x,y])\mid \psi(I)=I\} \, . \\
 		g & \mapsto & g^*
	\end{array}
\]
We then want to apply Corollary~\ref{Cor:ExtensionAutomorphisms} with $n=1$ and $h=xy-1$. To check that it is possible, we need to see that $h$ does not belong to $\k[w]$ for each variable $w\in\k[x,y]$. Indeed, $xy-1\in \k[w]$ would imply that $xy\in \k[w]$, and thus that $x,y\in \k[w]$, since $\k[w]$ is factorially closed (Lemma~\ref{Lemm:ProdKxKxy}). This would yield $\k[w]=\k[x,y]$, a contradiction.

We then apply Corollary~\ref{Cor:ExtensionAutomorphisms} and obtain a group isomorphism
 \[\begin{array}{rcl}
 \Aut_{\k[t]}(R)&\iso & \{\psi\in \Aut_{\k[t]}(\k[t,x,y])\mid \psi(I)=I\}\\
 \varphi & \mapsto & \varphi|_{\k[t,x,y]},\end{array}\] 
 where $R=\k[t,u,x,y]/(tu-xy-1)$. This yields then a group isomorphism
 \[\begin{array}{rcl}
\{g\in \Aut(\SL_2)\mid \pi\eta g=\pi\eta\}& \iso & \{g\in \Aut(\A^3_\k)\mid \pi g=\pi, g(\Gamma)=\Gamma\}\\
 g & \mapsto & \eta g \eta^{-1}.
\end{array}\]
It remains then to show that 
\[\{g\in \Aut(\SL_2)\mid \pi\eta g=\pi\eta\}=\{g\in \Aut(\SL_2)\mid g(X)=X\}.\]
The inclusion ``$\subset$'' follows from the equality $X=(\pi\eta)^{-1}(1)$. It remains then to show the inclusion ``$\supset$''.

To do this, we take $g\in \Aut(\SL_2)$ such that $g(X)=X$ and  prove that $\pi\eta g=\pi\eta$. The element $g$ corresponds to an element $g^{*}\in \Aut_\k(R)$. 
The fact that $g(X)=X$ is then equivalent to ask that $g^*$ sends the ideal generated by $t-1$ onto itself. Since $R^*=\k^*$ by Lemma~\ref{lem:R_is_UFD}, so $t-1$ is sent onto $\mu(t-1)$ for some $\mu\in \k^*$. This implies that the restriction of $g^*$ yields an automorphism of $\k[t]$, corresponding to an automorphism $\hat{g}\in \Aut(\A^1_\k)$ such that $\hat{g}\pi \eta=\pi \eta g$. As $(\pi\eta)^{-1}(0)$
is the only fibre of $\pi \eta$ that is not isomorphic to $\A^2_\k$, 
it has to be preserved under $g$. As the fibre $\pi^{-1}(1)=X$ is 
also preserved under $g$, we find that $\hat{g}$ is the identity, 
so $\pi\eta g=\pi\eta$ as desired.
\end{proof}
\begin{corollary}\label{Corollary:Embnu}The closed embedding \[
\begin{array}{cccc}
\nu\colon&\A^2_\k&\hookrightarrow& \SL_2\\
&(x,y)&\mapsto &\begin{pmatrix}
				  x & 1 \\
				  xy-1 & y
			     \end{pmatrix}\end{array}\] has the following property: an automorphism of $\A^2_\k$ extends to an automorphism of $\SL_2$, via $\nu$, if and only if it has Jacobian determinant equal to $\pm 1$.
\end{corollary}
\begin{proof}
We denote by $X=\nu(\A^2_\k)\subset \SL_2$ the closed hypersurface given by 
\[
	X=\nu(\A^2_\k)=\left.\left\{
			    \begin{pmatrix}
				  x & 1 \\
				  xy-1 & y
			     \end{pmatrix}\right| (x,y)\in \A^2_\k\right\}=\left\{\left. 
			     \begin{pmatrix}
				  x & t \\
				  u & y
			     \end{pmatrix}\in \SL_2 \right| t= 1 \right\},
\]
 write $G=\{g\in \Aut(\SL_2)\mid g(X)=X\}$ and denote by $\tau\colon G\to \Aut(\A^2_\k)$ the group homomorphism such that 
 $g\circ\nu=\nu\circ \tau(g)$ for each $g\in G$.

We first prove that the subgroup $H=\{h\in \Aut(\A^2_\k)\mid \Jac(h)\pm 1\}$ is contained in $\tau(G)$. The group $H$ is generated by $(x,y)\mapsto (y,x)$, which is induced by 
\[
	\begin{pmatrix}
				  x & t \\
				  u & y
			     \end{pmatrix}\mapsto 
			     \begin{pmatrix}
				  y & t \\
				  u & x
			     \end{pmatrix} \, ,
\] 
and by automorphisms of the form $(x,y)\mapsto (x,y+p(x))$, $p\in \k[x]$, induced by
\[
				\begin{pmatrix}
				  x & t \\
				  u & y
			     \end{pmatrix}\mapsto \begin{pmatrix}
				  1 & 0 \\
				  p(x) & 1
			     \end{pmatrix}\cdot\begin{pmatrix}
				  x & t \\
				  u & y
			     \end{pmatrix}.\]

			     It remains to take $g\in G$ and to prove that 
			     $\tau(g)\in H$. Proposition~\ref{AutoSL2fixingX} implies that $g$ can be written as
			     \[
			     		g
					\begin{pmatrix}
						x & t \\
						u & y
					\end{pmatrix}
					=
					\begin{pmatrix}
						a(t,x,y) & t \\
						s(t,u,x,y) & b(t,x,y)
					\end{pmatrix} \, ,
			     \]
			     where $a,b\in \k[t,x,y]$, $s\in \k[t,u,x,y]$ and such that 
			     $\tilde{g}\colon \A^2_{\k} \to \A^2_{\k}$
			     given by $\tilde{g}(t,x,y) = (t,a(t,x,y),b(t,x,y))$ 
			     is an automorphism of $\A^3_\k$ that preserves the 
			     curve $\Gamma$ given by $t=xy-1=0$.
			     The Jacobian determinant of $\tilde{g}$ is $\mu=\frac{\partial a}{\partial x}\cdot \frac{\partial b}{\partial y}-\frac{\partial a}{\partial y}\cdot \frac{\partial b}{\partial x}\in \k^*$. Replacing with $t=0$, we obtain the automorphism of $\A^2_\k$ given by $(x,y)\mapsto (a(0,x,y),b(0,x,y))$, which preserves the curve with equation $xy=1$ and is thus of Jacobian $\pm 1$. Indeed, it is of the form $(x,y)\mapsto (\xi x,\xi^{-1} y)$ or $(x,y)\mapsto (\xi y,\xi^{-1} x)$, for some $\xi\in \k^*$ (see \cite[Theorem~2~$(iii)$]{BlaSta}). This shows that $\mu=\pm 1$. Replacing then $t=1$ we get that the automorphism $\tau(g)$ which is given 
			     by $\tau(g)(x,y) = (a(1,x,y),b(1,x,y))$ has 
			     Jacobian $\pm 1$.
\end{proof}
\begin{proof}[Proof of Theorem~$\ref{Thm:StdEmbSL2}$]
We first observe that the embeddings $\rho_1,\bar\nu\colon \A^2_\k\to \SL_2$ given by
\[
\begin{array}{ccccccccccc}
\rho_1\colon&\A^2_\k&\hookrightarrow& \SL_2 && \bar\nu\colon & \A^2_\k&\hookrightarrow& \SL_2\\
&(a,b)&\mapsto &\begin{pmatrix}
				  1 & b \\
				  a & ab+1
			     \end{pmatrix}, &&&(a,b)&\mapsto &\begin{pmatrix}
				  a & 1 \\
				  -ab-1 & -b
			     \end{pmatrix},\end{array}\]
are equivalent, under the map
			    $\begin{pmatrix}
				  x & t \\
				  u & y
			     \end{pmatrix}\mapsto \begin{pmatrix}
				  u & x \\
				  -y & -t
			     \end{pmatrix}$. 
The embedding 
\[
	\begin{array}{cccc}
		\nu\colon&\A^2_\k&\hookrightarrow& \SL_2\\
		&(x,y)&\mapsto &\begin{pmatrix}
				  x & 1 \\
				  xy-1 & y
			     \end{pmatrix},
	\end{array}
\]
satisfies $\bar\nu =\nu \tau$, where $\tau\in \Aut(\A^2_\k)$ is the automorphism of Jacobian $-1$ given by $\tau\colon (x,y)\mapsto (x,-y)$. Corollary~\ref{Corollary:Embnu} then 
			     implies that there exists $\hat\tau\in \Aut(\SL_2)$ such that $\hat\tau\nu=\nu\tau=\overline\nu$, i.e.~that the embeddings $\overline{\nu}$ and $\nu$ are equivalent, so $\nu$ and $\rho_1$ are equivalent.
Corollary~\ref{Corollary:Embnu} implies then that an automorphism of $\A^2_\k$ extends to an automorphism of $\SL_2$, via $\rho_1$, if and only if it has Jacobian determinant equal to $\pm 1$. It remains to prove Assertions~(\ref{Thm:StdEmbSL2ass1}) and (\ref{Thm:StdEmbSL2ass2}) of Theorem~\ref{Thm:StdEmbSL2}.
			     
Assertion~(\ref{Thm:StdEmbSL2ass2}) follows from the fact that the
group homomorphism 
\[
	\Jac\colon \Aut(\A^2_\k)\to \k^*
\] 
is surjective (taking for instance diagonal automorphisms), so there are automorphisms of Jacobian determinant in $\k^*\setminus \{\pm 1\}$ if and only if $\k$ contains at least $4$ elements.

To obtain Assertion~(\ref{Thm:StdEmbSL2ass1}), we observe that every closed embedding $\A^2_\k\to \SL_2$ having image in $\rho_1(\A^2)$ is of the form $\rho_1 \nu$ for some $\nu\in \Aut(\A^2_\k)$. Writing $d_\lambda\in \Aut(\A^2_\k)$ the automorphism given by $d_\lambda\colon (s,t)\mapsto (\lambda s,t)$, $\lambda\in \k^*$, we can write $\nu= d_\lambda \nu_1$ for some $\lambda\in \k^*$ and some $\nu_1\in \Aut(\A^2_\k)$ of Jacobian determinant equal to $1$. The result above implies that $\rho_1 \nu$ is equivalent to $\rho_1 d_\lambda=\rho_\lambda$.
			     
It remains to observe that $\rho_{\lambda'}=\rho_{\lambda}  d_{\lambda'\lambda^{-1}}$, 
so $\rho_{\lambda}$ and $\rho_{\lambda'}$ are equivalent if
		and only if $\lambda'\lambda^{-1}\in \{\pm 1\}$, which corresponds to $\lambda' = \pm \lambda$.
\end{proof}

\begin{remark}
	\label{Rem:HolomorphicA2inSL2}
	Over the field $\k = \C$ of complex numbers, all algebraic embeddings
	of $\C^2$ into $\SL_2(\C)$ with image equal to $\rho_1(\C^2)$
	are equivalent under holomorphic automorphisms of $\SL_2(\C)$.
	Indeed, according to 
	Theorem~$\ref{Thm:StdEmbSL2}$$\eqref{Thm:StdEmbSL2ass1}$ 
	it is enough to show that the embeddings
	\[
		\begin{array}{ccccccccccc}
		\rho_1\colon&\C^2&\hookrightarrow& \SL_2 && \rho_\lambda \colon & 
		\C^2 &\hookrightarrow& \SL_2\\
			     &(s,t )&\mapsto &
			     \begin{pmatrix}
				  1 & t \\
				  s & st+1
			     \end{pmatrix}, 
			     &&&(s,t)&\mapsto &
			     \begin{pmatrix}
				  1 & t \\
				  \lambda s & \lambda st+1
			     \end{pmatrix} 
		\end{array}
	\]
	are equivalent under a holomorphic automorphism
	for all $\lambda \in \C^\ast$. Such a holomorphic
	automorphism of $\SL_2(\C)$ is given by
	\[
		\begin{pmatrix}
			x & t \\
			u & y
		\end{pmatrix}
		\mapsto
		\begin{pmatrix}
			x & t \\
			\mu(x) u & y + \frac{\mu(x)-1}{x} tu
		\end{pmatrix} \, ,
	\]
	where $\mu \colon \C \to \C^\ast$
	is a holomorphic function with $\mu(1) = \lambda$ and $\mu(0) = 1$. 
\end{remark}

\section{Fibered embeddings of $\A^2_\k$ into $\SL_2$ 
and the proof of Theorem~$\ref{thm:Fibered_Embeddings}$}
In this section, we study fibred embeddings (as in \ref{Eq:FibredEmb}).

We will need the following simple description of the morphism $\eta\colon \SL_2\to \A^3_\k$ already studied in Proposition~\ref{AutoSL2fixingX}:
			     
\begin{lemma}\label{Lemma:BlGammaA3}Let $\Gamma\subset \A^3_\k=\Spec(\k[t,x,y])$ be the curve given by $t=xy-1=0$ and let $\eta\colon \Bl_\Gamma(\A^3_\k)\to \A^3_\k$ the blow-up of $\Gamma$. We then have a natural open embedding $\SL_2\hookrightarrow \Bl_\Gamma(\A^3_\k)$ such that the restriction of $\eta$ corresponds to $(t, u, x, y) \mapsto (t,x,y)$.
\end{lemma}
\begin{proof}The blow-up of $\Gamma$ can be seen as
\[\begin{array}{cccccc}
\eta\colon &\Bl_\Gamma(\A^3_\k)&=&\{((t,x,y),[u:v])\in \A^3_\k\times \p^1_\k\mid tu=(xy-1)v\}&\to& \A^3_\k\\
& && ((t,x,y),[u:v]) & \mapsto & (t,x,y).\end{array}
\]
The open subset of $\Bl_\Gamma(\A^3_\k)$ given by $v\not=0$ is then naturally isomorphic to $\SL_2$, by identifying $((t,x,y),[u:1])$ with 
$(t, u, x, y) \in \SL_2$, and the birational morphism $\eta\colon \SL_2\to \A^3_\k$ sends $((t,x,y),[u:1])$ onto $(t,x,y)$.
\end{proof}

			     \subsection{Polynomials associated to fibred embeddings}

The following result associates to every fibred embedding $\A^2_\k\hookrightarrow \SL_2$ a polynomial in $\k[t,x,y]$, and gives some basic properties of this polynomial (which will be studied in more details after).

\begin{lemma}\label{Lemm:Polynomialp}
Let $\rho\colon \A^2_\k\hookrightarrow \SL_2$ be a fibred embedding, and let $Z\subset \A^3_\k$ be the closure of $\eta(\rho(\A^2_\k))$, where 
\[
	\eta\colon \SL_2\to \A^3_\k \, , \quad
				\begin{pmatrix}
				  x & t \\
				  u & y
			     \end{pmatrix}\mapsto (t,x,y) \, .
\] 
Then, $Z$ is given by $P(t,x,y)=0$, 
where $P\in \k[t,x,y]$ is a polynomial having the following properties:
			     
			     \begin{enumerate}[label=$(\arabic*)$, ref=\arabic*]
			     \item\label{TrivialfibrationOverOpen}
			     The ring $\k[t,\frac{1}{t}][x,y]/(P)$ is a polynomial ring in one variable over $\k[t,\frac{1}{t}]$ $($equivalently the morphism $\pi \colon Z\to \A^1_\k$ given by $(t,x,y)\mapsto t$ is a trivial $\A^1$-bundle over 
			     $\A^1_\k\setminus \{0\})$. 
			     \item\label{DegenerateFibre}
			     If $P$ is a variable of the $\k(t)$-algebra $\k(t)[x,y]$ 
			     $($which is always true if $\car(\k)=0$
			     by~\eqref{TrivialfibrationOverOpen} and the 
			     Abhyankar-Moh-Suzuki Theorem$)$, then the polynomial
			     $P(0,x,y)\in \k[x,y]$ is given by $\mu x^m(x-\lambda)$ or 
			     $\mu y^m(y-\lambda)$  
			     for some $\mu,\lambda\in \k^*$ 
			     and some $m\ge 0$, 
			     and $\rho(\A^2_\k)\subset \SL_2$ is the 
			     hypersurface given by $P=0$. 
			     \end{enumerate}
\end{lemma}

\begin{proof}
We consider the morphisms
\[
	\begin{array}{ccccc}
		\SL_2=\Spec(\k[t,u,x,y] / (xy-tu-1)) & \stackrel{\eta}{\to} & 
		\A^3_\k & \stackrel{\pi }{\to}& \A^1_\k\\
		(t,u,x,y)& \mapsto & (t,x,y) & \mapsto & t\end{array}
\]
and observe that $\eta$ yields an isomorphism between the two open subset $(\SL_2)_t\subset \SL_2$ and ${(\A^3_\k)_t} \subset \A^3_\k$ given by $t\not=0$. The morphism $\eta\rho\colon \A^2_\k \to \A^3_\k$ restricts thus to a closed embedding ${(\A^2_\k)_t} \hookrightarrow {(\A^3_\k)_t}$, where 
${(\A^2_\k)_t} \subset \A^2_\k$ is the open subset where $t\not=0$. This yields (\ref{TrivialfibrationOverOpen}).

We now assume that $P\in \k[t,x,y]$ is a variable of the $\k(t)$-algebra $\k(t)[x,y]$. Applying Lemma~\ref{Lemm:DegVariables} we obtain that $P_0=P(0,x,y)\in \k[w]$ for some variable $w\in \k[x,y]$. In particular, $xy-1$ does not divide $P_0$ (otherwise, by Lemma~\ref{Lemm:ProdKxKxy} we would have $xy-1\in \k[w]$ and then $x,y\in \k[w]$, impossible). This implies that $Z\cap \Gamma$ is a $0$-dimensional scheme
(which is a priori not reduced), where $\Gamma\subset \A^3_\k$ is the closed curve given by $t=xy-1=0$. 
Recall that $\SL_2$ is an open subset of $\Bl_\Gamma(\A^3_\k)$ (Lemma~\ref{Lemma:BlGammaA3}) and that the exceptional divisor $E\subset \SL_2$ is simply given by $t=xy-1=0$ and is a trivial $\A^1$-bundle over $\Gamma$. Since 
the pull-back $H\subset \SL_2$ of $Z$ on $\SL_2$, given by the equation $P=0$ 
has all its irreducible components of pure codimension $1$ we get
$H=\rho(\A^2_\k)\simeq \A^2_\k$. As $\rho$ is a fibred embedding, the morphism $H\to \A^1_\k$ given by the projection on $t$ is a trivial $\A^1$-bundle. 
This implies that $Z\cap \Gamma$ consists of a single reduced point, which
is defined over $\k$ and thus of the form $q=(0,\lambda,\frac{1}{\lambda})\in \Gamma$ for some $\lambda\in \k^*$.

We can thus write $P_0\in \k[w]$ as $P_0=ab$ where $a,b\in \k[w]$ are such that $b(q)\not=0$, $a$ is irreducible and $a(q)=0$. This implies that $a$ is a polynomial of degree $1$ in $w$, so we can assume that $a=w$ (by replacing $w$ with $a$).

We now show that $w=x-\lambda$ or $w=y-\frac{1}{\lambda}$ (after replacing $w$ with $\mu w$, $\mu\in \k^*$). As $w$ is a variable in $\k[x,y]$, the curve $C\subset \A^2_\k$ defined by $w=0$ is isomorphic to $\A_{\k}^1$, and its closure in $\p^2_\k$ is a curve $\overline{C}$ passing through exactly one point $q_0$ of the line at infinity $L_\infty=\p^2_\k\setminus \A^2_\k$.  The closure of $\Gamma$ is then $\overline{\Gamma}\subset \p^2_\k$ given by $xy-z^2=0$, and $\overline{\Gamma}\setminus \Gamma=\{[1:0:0],[0:1:0]\}$. To get $w=x-\lambda$ or $w=y-\frac{1}{\lambda}$ we then only need to show that $\overline{C}$ is a line through $[1:0:0]$ or $[0:1:0]$. We extend the scalars to an algebraically closed field and apply B\'ezout Theorem to get $2\deg(\overline{C})=1+\sum m_{q_0'}$ where the sum is taken over all points
$q_0'$ of $\overline{C}$ infinitely near to $q_0$ and lying on $\overline{\Gamma}$ 
and $m_{q_0'}$ denotes the multiplicity of $\overline{C}$ at $q_0'$ 
(follows from the fact that $C\cap \Gamma$ is the reduced point $q$). If $\overline{C}$ is a line, we get the result, since $q \in \overline{C}$. If $\overline{C}$ has degree at least $2$, then it is tangent to $L_\infty$ at the point $q_0$,
because, otherwise 
$m_{q_0} = (\overline{C}, L_{\infty})_{q_0} = \deg(\overline{C})$ by B\'ezout's Theorem and if $L$ denotes the tangent line of $\overline{C}$ at $q_0$,
then $\deg(\overline{C}) \geq (\overline{C}, L)_{q_0} \geq 1 + m_{q_0}$,
a contradiction.
As $L_\infty$ and $\overline{\Gamma}$ have transversal intersections, this implies that only $q_0$ belongs to $\overline{\Gamma}$ and yields 
$2\deg(\overline{C})-1 = m_{q_0} \le \deg(\overline{C})$, yielding $\deg(\overline{C})\le 1$ as desired.

Now that $w=x-\lambda$ is proven (respectively $w=y-\frac{1}{\lambda}$), we obtain $P_0=wb$ for some $b\in \k[x]$ (respectively $b\in \k[y]$) which does not vanish on any point of $\Gamma$. Hence, $P_0$ is equal to $x^m(x-\lambda)$ or $y^m(y-\frac{1}{\lambda})$ for some $m\ge 0$, after replacing $P$ with $\mu P$, $\mu\in \k^*$. 
\end{proof}
 
We now give an example which shows that the polynomial $P$ given in Lemma~\ref{Lemm:Polynomialp} is not always a variable of the $\k(t)$-algebra $\k(t)[x,y]$ (even if $P$ is always such a variable when $\car(\k)=0$).

\begin{lemma}\label{Lemm:ExampleNotVarkt}Let $\k$ be a field of characteristic $p>0$ and let $q \geq 2$ be an integer that does not divide $p$. Then, the polynomial \[P = (x-1)-t^p(y^p-(x-1)^q)^p\in \k[t,x,y]\] has the following properties:
	\begin{enumerate}[label=$(\arabic*)$, ref=\arabic*]
	\item\label{Pisnotvart}
	$P$ is not a variable of the $\k(t)$-algebra $\k(t)[x,y]$.
	\item\label{ZPtirivialA1b}
	The hypersurface $Z_P\subset \A^3_\k=\Spec(\k[t,x,y])$ 
	given by $P=0$ satisfies that $Z_P\to \A^1_\k$, $(t,x,y)\mapsto t$ 
	is a trivial $\A^1$-bundle $($in particular, $Z_P$ is isomorphic to $\A^2_\k)$.
	\item\label{HPtirivialA1b}
	The hypersurface $H_P\subset \SL_2=\Spec(\k[t,u,x,y] / (xy-tu-1))$ 
	given by $P=0$ is the image of a fibred embedding 
	$\A^2_\k\hookrightarrow \SL_2$ $($in particular, $H_P$ 
	is isomorphic to $\A^2_\k)$.
	\end{enumerate}
\end{lemma}

\begin{proof}(\ref{Pisnotvart}):
	Replacing $x$ with $x+1$, it suffices to show that $x-t^p(y^p-x^q)^p$ is not a variable of the $\k(t)$-algebra $\k(t)[x,y]$. This follows from Corollary~\ref{Cor:NotVariablesCharP}.

	(\ref{ZPtirivialA1b}) We consider the morphisms 
	\[
		\begin{array}{cccc}
		\tau\colon &\A^2_\k & \to & Z_P\\
		&(s,t)&\mapsto& (t, t^p s^{p^2}+1, t^q s^{pq}+s)\\
		\vspace{0.1cm}
	
		\chi\colon &Z_P & \to & \A^2_\k\\
		& (t,x,y) & \mapsto & (y-t^q(y^p-(x-1)^q)^q,t)\end{array}
	\]
	and check that $\tau \circ \chi=\mathrm{id}_{Z_P}$, 
	$\chi \circ \tau=\mathrm{id}_{\A^2_\k}$.
	
	(\ref{HPtirivialA1b}): The morphism $\eta\colon H_P\to Z_P$, $(t,u,x,y)\mapsto (t,x,y)$ being an isomorphism on the subsets given by $t\not=0$, the morphism 
	$\pi \circ \eta \colon H_P\to \A^1_{\k}$, $(t,u,x,y)\mapsto t$ is a trivial $\A^1$-bundle over $\A^1_{\k} \setminus \{0\}$. The zero fibre is moreover isomorphic to $\A^1_\k$ since $P(0,x,y)=x-1$ and the line $\{x=1\}$ intersects the conic 
	$\{xy=1\}$ transversally in one point 
	(follows from Lemma~\ref{Lemma:BlGammaA3}). 
	By Lemma~\ref{Lem:BCWdim1} it follows that 
	$\pi \circ \eta \colon H_P \to \A^1_{\k}$ is a trivial $\A^1$-bundle. 
	Hence $H_P$ is isomorphic to $\A^2_{\k}$
	and  is the image of a fibred embedding 
	$\A^2_{\k} \hookrightarrow \SL_2$.
\end{proof}

We now start from a polynomial $P\in \k[t,x,y]$ that is a variable of the $\k(t)$-algebra $\k(t)[x,y]$ and determine when this one comes from a fibred embedding $\A^2_\k\hookrightarrow \SL_2$, by the process determined in Lemma~\ref{Lemm:Polynomialp}. This yields the following result, which corresponds to Part~(\ref{ThmFibered1}) of Theorem~\ref{thm:Fibered_Embeddings}.

\begin{proposition}\label{Prop:LiftPAssertions}
Let $\k$ be any field, let $P\in \k[t,x,y]$ be a polynomial that is a variable of the $\k(t)$-algebra $\k(t)[x,y]$, and let $H_P\subset \SL_2=\Spec(\k[t,u,x,y] / (xy-tu-1))$ and $Z_P\subset \A^3_\k=\Spec(\k[t, x, y])$ be the hypersurfaces 
given by $P=0$.

The following conditions are equivalent:
\begin{enumerate}[label=$(\alph*)$, ref=\alph*]
\item\label{ThmEqa}
The hypersurface $H_P\subset \SL_2$ is isomorphic to $\A^2_\k$.
\item\label{ThmEqb}
The hypersurface $H_P\subset \SL_2$ is the image of a fibred embedding $\A^2_\k\hookrightarrow \SL_2$.
\item\label{ThmEqc}
The fibre
of $Z_P \to \A^1_{\k}$, $(t, x, y) \mapsto t$ over every closed point 
of $\A_{\k}^1\setminus \{0\}$
is isomorphic to $\A^1$ and the 
polynomial $P(0,x,y)\in \k[x,y]$ is of the form  $\mu x^m(x-\lambda)$ or 
			     $\mu y^m(y-\lambda)$  
			     for some $\mu,\lambda\in \k^*$  
			     and some $m\ge 0$. 
			
\end{enumerate}
\end{proposition}

\begin{proof}
We will use the morphisms 
\[
	\eta\colon \SL_2\to \A^3_\k \, , \quad
				\begin{pmatrix}
				  x & t \\
				  u & y
			     \end{pmatrix}\mapsto (t,x,y),\quad \pi\colon \A^3_\k\to \A^1_\k\, , \quad (t,x,y)\mapsto t \, .
\]

	$\eqref{ThmEqa} \Rightarrow \eqref{ThmEqb}$: 
	Proving that $H_P$ is the image of a fibred embedding 
	$\A^2_\k \hookrightarrow \SL_2$ is equivalent to ask that
	$\pi \circ \eta \colon H_P \to\A^1_\k$ is a trivial $\A^1$-bundle.
	Since $P$ is a variable of the $\k(t)$-algebra
	$\k(t)[x, y]$, it follows that the generic fibre of $\pi \colon Z_P \to \A^1_{\k}$
	is isomorphic to $\A^1_{\k(t)}$.
	Moreover, $\eta \colon \SL_2 \to \A^3_\k$ is an isomorphism
	over $\{ t \neq 0 \}$, so the generic fibre of $\pi \circ \eta$
	is also isomorphic to $\A^1_{\k(t)}$.
	The fact that $H_P$ is isomorphic to $\A^2_\k$ 
	(which is the hypothesis~\eqref{ThmEqa}) implies that 
	$\pi \circ \eta\colon H_P\to \A^1_\k$
	is a trivial $\A^1$-bundle, by
	Lemma~\ref{Lem:PolynomTwoVariables} 
	($\eqref{Pktxyfield} \Rightarrow \eqref{Ptrivial}$).
	
	$\eqref{ThmEqb} \Rightarrow \eqref{ThmEqc}$: Follows from 
	Lemma~\ref{Lemm:Polynomialp}\eqref{TrivialfibrationOverOpen} and 
	\eqref{DegenerateFibre}. 
	
	$\eqref{ThmEqc} \Rightarrow \eqref{ThmEqa}$: Since
	$\eta \colon \SL_2 \to \A^3_{\k}$ is an isomorphism over the open subset
	$\{ t \neq 0 \}$, it follows that all fibres of
	$\pi \circ \eta\colon H_P \to \A^1_\k$ over closed points 
	of $\A_{\k}^1\setminus\{0\}$
	are isomorphic to $\A^1$. Moreover, the fibre
	of $\pi \circ \eta$ over $0$ is isomorphic to $\A_{\k}^1$,
	since the restriction 
	$\eta |_{\{t=0\}} \colon \{ t = 0 \} \to \{ t = xy - 1 = 0 \} \subset 
	\{ 0 \} \times \A^2_\k$ 
	is a trivial $\A^1$-bundle over the curve $\{ t = xy - 1 = 0 \}$
	and since $\{ P(0, x, y) = 0 \}$ intersects $\{ xy = 1 \}$ in exactly one 
	point, transversally. The generic fibre of $\pi \circ \eta \colon H_P \to \A^1_\k$
	being isomorphic to $\A^1_{\k(t)}$, it follows from Lemma~\ref{Lem:BCWdim1} 
	that $\pi \circ \eta\colon H_P\to \A^1_\k$ is a trivial $\A^1$-bundle 
	and thus $H_P$ is isomorphic to the affine plane $\A^2_{\k}$, 
	which proves \eqref{ThmEqa}.
\end{proof}

\begin{example}\label{LargeFamilySimpleExamples}
For each $n\ge 1,m\ge 0$, $\mu\in \k^*$ and $q\in \k[t,x]$, the polynomial 
\[
	P(t,x,y)=t^n y+\mu x^m(x-1)+tq(t,x)\in \k[t,x,y]
\]
defines an hypersurface $H_{P}\subset \SL_2$ which is the image of a fibred embedding. Indeed, since $P$ has degree $1$ in $y$ with coefficent $t^n$, it is a variable of $\k[t,t^{-1}][x,y]$. We can thus apply Proposition~\ref{Prop:LiftPAssertions} and only need to check that $P(0,x,y)=\mu x^m(x-1)$ 
is of the desired form (as in Assertion~\eqref{ThmEqc}).
\end{example}

\subsection{Determining when two fibred embeddings are equivalent}
In this section, we consider embeddings satisfying the conditions of 
Proposition~\ref{Prop:LiftPAssertions} (or equivalently of Theorem~\ref{thm:Fibered_Embeddings}(\ref{ThmFibered1})) and determine when two of these are equivalent, by proving Theorem~\ref{thm:Fibered_Embeddings}(\ref{ThmFibered2}). We first characterise the case where the integer $m$ of Proposition~\ref{Prop:LiftPAssertions} (or equivalently of Lemma~\ref{Lemm:Polynomialp} or Theorem~\ref{thm:Fibered_Embeddings}(\ref{ThmFibered1})) is equal to zero.

\begin{lemma}\label{Lem:M0}
Let $\k$ be any field and $P\in \k[t,x,y]$ be a polynomial that is a variable of the $\k(t)$-algebra $\k(t)[x,y]$, and let $H_P\subset \SL_2=\Spec(\k[t,u,x,y] / (xy-tu-1))$ and $Z_P\subset \A^3_\k=\Spec(\k[t, x, y])$ 
be the hypersurfaces given by $P=0$.

Assume that $H_P$ is isomorphic to $\A^2_\k$, which implies that $P(0,x,y)\in \k[x,y]$ is of the form  $\mu x^m(x-\lambda)$ or $\mu y^m(y-\lambda)$ for some $\mu,\lambda\in \k^*$ and some $m\ge 0$. Then, the following conditions are equivalent:
			     \begin{enumerate}[label=$(\alph*)$, ref=\alph*]
			     \item\label{m0}
			     $m=0$;
			     \item\label{Pvarkt}
			     $P$ is a variable of the $\k[t]$-algebra $\k[t][x,y]$;
			     \item\label{ZA2fibr}
			     There is an isomorphism $\varphi\colon \A^2_\k\iso Z_P$ 
			     such that $\pi\varphi$ is the projection $(t, x)\mapsto t$.
			     \item\label{ZA2}
			     There is an isomorphism $\varphi\colon \A^2_\k\iso Z_P$.
			     \item\label{AutoStd}
			     There exist $\varphi\in \Aut(\SL_2)$  such that 
			     $\varphi(H_P)=\rho_1(\A^2_{\k})$, where $\rho_1$ 
			     is the standard embedding;
			     \item\label{AutofibStd}
			     There exist $\varphi\in \Aut(\SL_2)$  such that 
			     $\varphi(H_P)=\rho_1(\A^2_{\k})$ and 
			     $\varphi^*(t)=t$.
			     \end{enumerate}
\end{lemma}

\begin{proof} 
As before, we use the morphisms 
\[
	\eta\colon \SL_2\to \A^3_\k \, , \quad
				\begin{pmatrix}
				  x & t \\
				  u & y
			     \end{pmatrix}\mapsto (t,x,y),\quad \pi\colon \A^3_\k\to \A^1_\k\, , \quad (t,x,y)\mapsto t \, .
\]
Proposition~\ref{Prop:LiftPAssertions} says that $H_P\subset \SL_2$ is the image of a fibred embedding $\A^2_\k\hookrightarrow \SL_2$, which corresponds to say that $\pi\eta\colon H_P\to \A^1_\k$ is a trivial $\A^1$-bundle. Since
	$\eta \colon \SL_2 \to \A^3_{\k}$ is an isomorphism over the open subset
	$\{ t \neq 0 \}$, we obtain that $\pi \colon Z_P\to \A^1_\k$ is a trivial 
	$\A^1$-bundle over $\A^1_\k\setminus \{0\}$.

We first prove $(\ref{m0})\Leftrightarrow (\ref{Pvarkt})\Leftrightarrow (\ref{ZA2fibr})\Leftrightarrow (\ref{ZA2})$, using Corollary~\ref{Cor:ZinA3}. We observe that $(\ref{Pvarkt})$, $(\ref{ZA2fibr})$ and $(\ref{ZA2})$ correspond respectively to the equivalent assertions $(\ref{PPVarkt})$, $(\ref{PPisoA2fibr})$ and $(\ref{PPisoA2})$ of Corollary~\ref{Cor:ZinA3}. Moreover, the condition $m=0$ (which is $(\ref{m0})$) corresponds to say that the $0$-fibre of $\pi \colon Z_P\to \A^1_\k$ is isomorphic to $\A^1_\k$. Since $\pi \colon Z_P\to \A^1_\k$ is a trivial 
$\A^1$-bundle over $\A^1_\k\setminus \{0\}$, assertion~$(\ref{m0})$ corresponds to assertion $(\ref{PPeachfibA1})$ of Corollary~\ref{Cor:ZinA3}. Thus Corollary~\ref{Cor:ZinA3} yields 
\[
	(\ref{m0})\Leftrightarrow (\ref{Pvarkt})\Leftrightarrow (\ref{ZA2fibr})
	\Leftrightarrow (\ref{ZA2}).
\]
It remains to show that these are also equivalent to $(\ref{AutoStd})$ and $(\ref{AutofibStd})$.

$(\ref{Pvarkt})\Rightarrow (\ref{AutofibStd})$: Applying an automorphism of the form
\[
			     \begin{pmatrix}
				  x & t \\
				  u & y
			     \end{pmatrix}\mapsto\begin{pmatrix}
				  \mu^{-1}x & t \\
				  u & \mu y
			     \end{pmatrix}
			     \quad \textrm{or} \quad 
			     \begin{pmatrix}
				  x & t \\
				  u & y
			     \end{pmatrix}\mapsto
			     \begin{pmatrix}
				  \mu^{-1}y & t \\
				  u & \mu x
			     \end{pmatrix}
\]
for some $\mu\in \k^*$, we can assume that $P(0,x,y)=x-1$. Since $P$ is a variable of the $\k[t]$-algebra $\k[t][x,y]$, there exists $f\in \Aut_{\k[t]}(\k[t,x,y])$ such that $f(x-1)=P$. The element $\psi\in \Aut(\A^3_\k)$ satisfying $\psi^*=f$ is then such that $\pi\psi=\pi$ and sends $Z_P$ onto the hypersurface of $\A^3_\k$ given by $x=1$. The restriction of $\psi$ to the hypersurface given by $t=0$ is an automorphism of the form $(0,x,y)\mapsto (0,\nu(x,y),\rho(x,y))$ which preserves the curve given by $x-1=0$. Replacing $\psi$ with its composition with
the inverse of $(t,x,y)\mapsto (t,\nu(x,y),\rho(x,y))$, we can assume that the restriction of $\psi$ to the hypersurface $t=0$ is the identity, so $\psi(\Gamma)=\Gamma$, where $\Gamma$ is the curve given by $t=xy-1=0$.  Proposition~\ref{AutoSL2fixingX} implies then that $\psi$ lifts to an automorphism $\varphi$ of $\SL_2$ sending $H_P$ onto $\rho_1(\A^2_{\k})$. 
We moreover have $\varphi^*(t)=t$, since $\psi^*(t)=t$.

$(\ref{AutofibStd})\Rightarrow (\ref{AutoStd})$ being clear, it 
remains to show $(\ref{AutoStd})\Rightarrow(\ref{m0})$. 
For this implication, one can assume that $\k$ is algebraically closed. 
Assertion~$(\ref{AutoStd})$ yields an automorphism $\varphi \in \Aut(\SL_2)$
such that $\varphi(H_P) = \rho_1(\A^2_{\k})$. Hence the automorphism 
$\varphi^*\in\Aut_\k(R)$, where $R=\Spec(\k[t,u,x,y] / (xy-tu-1))$, sends the 
ideal $(x-1)\subset R$ onto the ideal $(P)\subset R$. It follows from 
Lemma~\ref{lem:R_is_UFD} that $x-1$ is sent onto $\mu P$, for some 
$\mu\in \k^*$. In particular, for a general $a\in \k$, the variety 
$H_{P-a} \subset \SL_2$ given by $P-a=0$ is isomorphic to $\A^2_\k$. 
It remains to show that this implies that $m=0$.
	
	Since $P$ is a variable of the $\k(t)$-algebra $\k(t)[x, y]$, so is $P-a$. 
	There exists then an open dense subset $U \subset \A^1_{\k}$ 
	such that $U \times \A^2_{\k} \to U \times
	\A^1_{\k}$, $(t, x, y) \mapsto (t, P(t, x, y) - a)$
	is a trivial $\A^1$-bundle. This implies that 
	$q_a \colon H_{P-a} \to \A^1_{\k}$ $(t, u, x, y) \mapsto t$
	is a trivial $\A^1$-bundle 
	over $U$. By Lemma~\ref{Lem:PolynomTwoVariables}, 
	$q_a$ is a trivial $\A^1$-bundle (since $H_{P-a} \simeq \A^2_\k$), 
	so the fibre $(q_a)^{-1}(\{0\})$ needs to be isomorphic to an affine line. 
	Since $(q_a)^{-1}(\{0\})$
	is given by the equations $xy-1=P(0, x,y)-a =t =0$
	in the affine $4$-space $\A^4_{\k} = \Spec(\k[t, u, x, y])$ and since
	$P(0,x,y)-a$ is equal to $\mu x^m(x-\lambda)-a$ or 
	$\mu y^m(y-\lambda)-a$ and $a\in \k$ is general 
	($\k$ is algebraically closed),
	this implies that $m=0$ and yields $(\ref{m0})$ as desired.
\end{proof}

\begin{remark}
\label{Rem:M0}
Lemma~\ref{Lem:M0} shows in particular that if $H_P,H_Q\subset \SL_2$ are two hypersurfaces given by two polynomials $P,Q\in \k[t,x,y]$ as in Theorem~\ref{thm:Fibered_Embeddings} (or as in the previous results), and if one of the two integers $m,m'\in \N$ associated to $P,Q$ is equal to zero, then $H_P,H_Q$ are equivalent if and only if $m=m'=0$.
\end{remark}

\begin{proposition} \label{prop:SL2eqthenA3eq}
Let $\k$ be any field, let $P,Q\in \k[t,x,y]$ be polynomials that are variables of the $\k(t)$-algebra $\k(t)[x,y]$, and let $H_P,H_Q\subset \SL_2=\Spec(\k[t,u,x,y] / (xy-tu-1))$ and $Z_P,Z_Q\subset \A^3_\k=\Spec(\k[t,u,x])$ be the hypersurfaces given by $P=0$ and $Q=0$ respectively.

Suppose that $H_P$ is isomorphic to $\A^2_\k$ but that $Z_P$ is not isomorphic to $\A^2_\k$, and that there exists $\varphi\in \Aut(\SL_2)$ that sends $H_P$ onto $H_Q$. Then, the following hold:
\begin{enumerate}[label=$(\arabic*)$, ref=\arabic*]
\item\label{varphistartt} There exists $\mu\in \k^*$ such that $\varphi^*(t)=\mu t$.
\item\label{PsiIsAlift}
The birational map $\psi=\eta\varphi\eta^{-1}$ is an automorphism of $\A^3_\k$ which sends $Z_P$ onto $Z_Q$, where $\eta\colon \SL_2\to \A^3_\k$ is as before given by $(t,u, x, y)\mapsto (t,x,y)$.
\item\label{SameM}
There exists $m\ge 1$ such that $P(0,x,y)$ and $Q(0,x,y)$ are of the form  $\mu x^m(x-\lambda)$ or 
			     $\mu y^m(y-\lambda)$  
			     for some $\mu,\lambda\in \k^*$ 
			     $($the integer $m$ is the same for $P,Q$ but 
			     $\mu,\lambda$ and the choice between $x$ and $y$ 
			     depend on $P,Q)$.
\end{enumerate}
\end{proposition}

\begin{proof}
Since $H_P$ is isomorphic to $\A^2_\k$, the same holds for $H_Q$. The hypersurfaces $H_P,H_Q\subset \SL_2$ are thus image of fibred embeddings $\A^2_\k\hookrightarrow \SL_2$ and there are thus integers $m,m'\ge 0$ and $\lambda,\lambda',\mu',\mu\in \k^*$ such that $P(0,x,y)\in \{\mu x^m(x-\lambda),\mu y^m(y-\lambda)\}$ and $Q(0,x,y)\in \{\mu' x^{m'}(x-\lambda'),\mu' y^{m'}(y-\lambda')\}$ (Proposition~\ref{Prop:LiftPAssertions}). Moreover, the fact that $Z_P$ is not isomorphic to $\A^2_\k$ is equivalent to $m>0$ and to the fact that $H_P$ is not equivalent to the image  $\rho_1(\A^2)$ of the standard embedding (Lemma~\ref{Lem:M0}). As $H_P$ and $H_Q$ are equivalent, the same hold for $H_Q$, so $m'>0$.

The main part of the proof consists in proving (\ref{varphistartt}). To do this, one 
can extend the scalars and assume $\k$ to be algebraically closed. 
We moreover have $\varphi^*(Q)=\xi P$ for some $\xi\in \k^*$ 
(follows from Lemma~\ref{lem:R_is_UFD}). Replacing $P$ with $\xi P$, 
we can assume that $\varphi^*(Q)=P$. For each $a\in \k^*$, the element 
$\varphi$ then sends $H_{P-a}$ onto $H_{Q-a}$, where 
$H_{P-a},H_{Q-a}\subset \SL_2$ are given by the polynomials 
$P-a,Q-a\in \k[t,x,y]$. If $a$ is chosen general, then $H_{P-a},H_{Q-a}$ 
are smooth hypersurfaces of $\SL_2$ (since this is true for $a=0$), and the 
same holds for the hypersurfaces $Z_{P-a},Z_{Q-a}\subset \A^3_\k$ given  by 
$P-a$ and $Q-a$ respectively. Since $P,Q$ are variables of the $\k(t)$-algebra 
$\k(t)[x,y]$ and because the $t$-projections $Z_P\to \A^1$ and 
$Z_Q\to \A^1$ are trivial $\A^1$-bundles over $\A_{\k}^1 \setminus \{ 0 \}$ 
(Lemma~\ref{Lemm:Polynomialp}\eqref{TrivialfibrationOverOpen}), 
the polynomials $P,Q$ are also
variables of the $\k[t,\frac{1}{t}]$-algebra $\k[t,\frac{1}{t}][x,y]$ 
(follows from Lemma~\ref{Lemm:ZinA3U} with $U=\A^1\setminus \{0\}$). 
Hence, the same holds for $P-a$ and $Q-a$. The morphisms 
$H_{P-a},H_{Q-a},Z_{P-a},Z_{Q-a}\to \A^1_\k$ given by the projection on 
$t$ are therefore trivial 
$\A^1$-bundles over $\A^1_\k\setminus \{0\}$. We can thus see these varieties 
as open subsets of smooth projective surfaces $\overline{H_{P-a}},\overline{H_{Q-a}},\overline{Z_{P-a}},\overline{Z_{Q-a}}$ obtained by blowing-up some Hirzebruch surfaces, so that the projection on $t$ is the restriction of the morphism to $\p^1_\k$ given by a $\p^1$-bundle of the Hirzebruch surface and having only one singular fibre. We can moreover assume that the boundary is a union of smooth rational curves of self-intersection $0$ or $\le -2$ (in particular the projectivisation is minimal). Indeed, if a component of the singular fibre has self-intersection $-1$ and is in the boundary, we can contract it, and if the section has self-intersection $-1$, then we blow-up a general point of the smooth fibre contained in the boundary and then contract the strict transform of this fibre to obtain a section of self-intersection $0$. The zero fibre of $Z_{P-a} \to \A^1_\k$ is given by $t=\mu x^m(x-\lambda)-a=0$ or $t=\mu y^m(y-\lambda)-a=0$ and is thus a disjoint union 
$C \simeq \coprod_{i=1}^{m+1} \A^1_{\k}$ 
of $m+1$ affine curves isomorphic to $\A^1_{\k}$. Similarly the zero fibre of $Z_{Q-a} \to \A^1_\k$ is a disjoint union 
$C' \simeq \coprod_{i=1}^{m'+1} \A^1_{\k}$ 
of $m'+1$ affine curves isomorphic to $\A^1_{\k}$. The closure of $C$ is  contained in the singular fibre $F_0$ of $\overline{Z_{P-a}}\to \p^1_\k$, which is a tree of smooth rational curves of self-intersection $\le -1$, being a SNC divisor. Hence, the closure of each component of $C$ is a smooth rational curve of self-intersection $\le -1$, which intersects the boundary into a component lying in $F_0$. A similar description holds for $C'$.

The curves $C$, $C'$ meet transversally the conic $\Gamma$ given by $xy=1$ (because of the form of $P(0,x,y)-a$ and $Q(0,x,y)-a$). 
The surfaces $\overline{H_{P-a}}, \overline{H_{Q-a}}$ 
are then obtained by blowing-up some 
points in each of the components of $C$, $C'$
and removing these components, so we can choose the a minimal 
projectivisations of $H_{P-a},H_{Q-a}$ to be blowing-ups of the above points in $\overline{Z_{P-a}},\overline{Z_{Q-a}}$ and get a dual graph of 
the boundary of these surfaces which is not a chain (or which is not ``linear'' or not a ``zigzag''). This implies that the $\A^1$-fibration given by the $t$-projection is unique up to automorphisms of the target (\cite[Theor\`eme~1.8]{Bertin}). As the zero fibre of 
$H_{P-a}, H_{Q-a} \to \A^1_{\k}$ is the unique degenerate fibre, 
there exist $\mu_a\in \k^*$ and $q_a\in \k[t,u,x,y]$ such that $\varphi^*(t)=\mu_a t+q_a\cdot (P-a)$. Since this holds for a general $a$, we get $\varphi^*(t)=\mu t$ for some $\mu\in \k^*$. Indeed, replacing $t$ with $0$ in $\varphi^*(t)$ yields an element of $\k[u,x,y]/(xy-1)$ which is divisible by $P-a$ for infinitely many $a$. 
This element is thus equal to zero.

We now show how Assertion~(\ref{varphistartt}) implies the two others. We write $\varphi=\varphi_1\varphi_2$ where $(\varphi_1)^*(t)=t$ and $\varphi_2$ is given by \[
			     \begin{pmatrix}
				  x & t \\
				  u & y
			     \end{pmatrix}\mapsto\begin{pmatrix}
				  x & \mu t \\
				  \mu^{-1}u &  y
			     \end{pmatrix}.
\] 
The fact that $(\varphi_1)^*(t)=t$ implies that $\psi_1=\eta\varphi_1\eta^{-1}$ is an automorphism of $\A^3_\k$ (Proposition~\ref{AutoSL2fixingX}). Since $\psi_2=\eta\varphi_2\eta^{-1}$ is a diagonal automorphism of $\A^3_\k$, the element $\psi=\psi_1\psi_2=\eta\varphi\eta^{-1}$ is an automorphism of $\A^3_\k$. As $\varphi$ sends $H_P$ onto $H_Q$, the automorphism $\psi$ sends $Z_P$ onto $Z_Q$, which yields~(\ref{PsiIsAlift}). As $\psi^*(t)=\mu t$, the hyperplane $W\subset \A^3_\k$ given by $t=0$ is invariant, this implies that $m=m'$ and thus yields (\ref{SameM}).
\end{proof}

Lemma~\ref{Lem:M0} and Proposition~\ref{prop:SL2eqthenA3eq} yield then the following result, which yields in particular Assertion~\eqref{ThmFibered2} of Theorem~\ref{thm:Fibered_Embeddings}:

\begin{corollary}\label{Coro:HPHQequiv}
If $P, Q \in \k[t,x,y]$ are polynomials which are variables of the 
	$\k(t)$-algebra $\k(t)[x,y]$ and if the corresponding
	hypersurfaces $H_P,H_Q\subset\SL_2=\Spec(\k[t,u,x,y] / (xy-tu-1))$ 
	are equivalent 
	and isomorphic to $\A^2_{\k}$, the following hold:
	\begin{enumerate}[label=$(\arabic*)$, ref=\arabic*]
	\item\label{HPQfibred}
	$H_P,H_Q$ are the image of fibred embeddings 
	$\A^2_\k\hookrightarrow \SL_2$.
	\item\label{HPQeqfibred}
	There exists $\varphi\in \Aut(\SL_2)$ such that $\varphi(H_P)=H_Q$ 
	and $\varphi^*(t)=\mu t$ for some $\mu\in \k^*$. 
	In particular, the element $\psi=\eta\varphi\eta^{-1}\in \Aut(\A^3_\k)$ satisfies 
	$\psi^*(t)=\mu t$, $\psi(Z_P)=Z_Q$ and $\psi(\Gamma) = \Gamma$, where 
	$\eta\colon \SL_2\to \A^3_\k$ is the morphism 
	$(t, u, x, y) \mapsto (t,x,y)$,
	$Z_P,Z_Q\subset \A^3_\k$ are the two hypersurfaces given by $P=0$, $Q=0$
	and $\Gamma\subset \A^3_\k$ is the conic given 
	by $t=xy-1=0$. 
	\end{enumerate} 
\end{corollary}
\begin{proof}
Assertion~\eqref{HPQfibred} follows from Proposition~\ref{Prop:LiftPAssertions}. It remains then to show \eqref{HPQeqfibred}. We denote by $\varphi_0\in \Aut(\SL_2)$ an element such that $\varphi_0(H_P)=H_Q$. 

$(i)$ If $\varphi_0^*(t)=\mu t$ for some $\mu\in \k^*$, we choose $\varphi=\varphi_0$ and denote by $\theta\in \Aut(\SL_2)$ the element 
\[			     
			     \begin{pmatrix}
				  x & t \\
				  u & y
			     \end{pmatrix}\mapsto \begin{pmatrix}
				  x & \mu^{-1} t \\
				  \mu u & y
			     \end{pmatrix}
\] 
to obtain $(\varphi_0\theta)^*(t) = t$. Proposition~\ref{AutoSL2fixingX} shows that 
$\hat\psi=\eta(\varphi_0\theta)\eta^{-1}\in \Aut(\A^3_\k)$ and 
$\hat\psi(\Gamma)=\Gamma$. Since $\tilde\theta=\eta\theta\eta^{-1}$ is the
automorphism of $\A^3_\k$ given by $(t,x,y)\mapsto (\mu^{-1} t,x, y)$, we have 
$\psi=\eta\varphi_0\eta^{-1}\in \Aut(\A^3_\k)$ and $\psi(\Gamma)=\Gamma$.  The 
fact that $\varphi_0^*(t)=\mu t$ and $\varphi_0(H_P)=H_Q$ yields then 
$\psi^*(t)=\mu t$ and $\psi(Z_P)=Z_Q$.

$(ii)$ If $\varphi_0^*(t)\not\in \{\mu t\mid \mu \in\k^*\}$, then Proposition~\ref{prop:SL2eqthenA3eq}(\ref{varphistartt}) does not hold, so $Z_P$ is isomorphic 
to $\A^2_\k$. Applying the same argument to $\varphi_0^{-1}$ shows that $Z_Q$ is isomorphic to $\A^2_\k$. Lemma~$\ref{Lem:M0}(\eqref{ZA2}\Rightarrow \eqref{AutofibStd})$ then shows that there exist $\varphi_1,\varphi_2\in \Aut(\SL_2)$ 
such that $\varphi_1(H_P)=\varphi_2(H_Q)=\rho_1(\A^2_{\k})$ 
and $(\varphi_1)^*(t)=(\varphi_2)^*(t)=t$. We then choose 
$\varphi=(\varphi_2)^{-1}\varphi_1$ and apply case $(i)$.
\end{proof}

\subsection{Examples of non-equivalent embeddings}

\begin{lemma}
\label{Lemm:InfiniteFamily}
To each polynomial $r\in \k[t]$, we associate the polynomial
\[P_r=ty-(x-t)(x-1-t^2r(t))\in \k[t,x,y]\]
and denote by let $H_{P_r}\subset \SL_2=\Spec(\k[t,u,x,y] / (xy-tu-1))$ and $Z_{P_r}\subset \A^3_\k=\Spec(\k[t, x, y])$ the hypersurfaces given by $P_r=0$. Then,
\begin{enumerate}[label=$(\arabic*)$, ref=\arabic*]
\item\label{HPrfibredemb}
For each $r\in \k[t]$, the surface $H_{P_r}$ is the image of a fibred embedding $\A^2_\k\hookrightarrow \SL_2$.
\item\label{EquivRS}
For each $r,s\in \k[t]$, the following are equivalent
\begin{enumerate}[label=$(\roman*)$, ref=\roman*]
\item\label{EquivRSHeq}
There exists $\varphi\in \Aut(\SL_2)$ such that $\varphi(H_{P_r})=H_{P_s}$.
\item \label{EquivRSZeq}
There exists $\varphi\in \Aut(\A^3_\k)$ such that $\varphi(Z_{P_r})=Z_{P_s}$.
\item\label{EquivRSZiso}
The surfaces $Z_{P_r}$ and $Z_{P_s}$ are isomorphic.
\item\label{EquivRSequal}
$r=s$.
\end{enumerate}
\end{enumerate}
\end{lemma}

\begin{proof}For each $r\in \k[t]$, we write $S_r(t,x)=(x-t)(x-1-t^2r(t))\in \k[t,x]$ and observe that $P_r(t,x,y)=ty-S_r(t,x)$. 

(\ref{HPrfibredemb}): Since $P_r$ is of degree $1$ in $y$, it is a variable of 
the $\k(t)$-algebra $\k(t)[x,y]$. Moreover, $P_r(0,x,y)=S_r(0,x)=x(x-1)$ is of the  form  $\mu x^m(x-\lambda)$ (with $\mu,\lambda\in \k^*$ and $m\ge 0$).
The coefficient of $y$ in $P_r$ being $t$, the morphism $Z_{P_r} \to \A^1_{\k}$, $(t, x, y) \mapsto t$ is a trivial $\A^1$-bundle over $\A^1_{\k} \setminus \{0\}$. Proposition~\ref{Prop:LiftPAssertions} ($\eqref{ThmEqc}\Rightarrow\eqref{ThmEqb}$) then implies that $H_{P_r}\subset \SL_2$ is the image of a fibred embedding $\A^2_\k\hookrightarrow \SL_2$.

It remains to show that the assertions 
$(\ref{EquivRSHeq})-(\ref{EquivRSZeq})-(\ref{EquivRSZiso})-(\ref{EquivRSequal})$ of $(\ref{EquivRS})$ are equivalent. 

The implications $(\ref{EquivRSequal})\Rightarrow(\ref{EquivRSHeq})$ and $(\ref{EquivRSZeq})\Rightarrow (\ref{EquivRSZiso})$ are trivial.

Lemma~\ref{Lem:M0} implies that $Z_{P_r}$ and $Z_{P_s}$ are not isomorphic to $\A^2_\k$ (the integer $m$ being here equal to $1$). We can thus appply Proposition~\ref{prop:SL2eqthenA3eq}(\ref{PsiIsAlift}), which yields $(\ref{EquivRSHeq})\Rightarrow (\ref{EquivRSZeq})$. 

It remains then to show $(\ref{EquivRSZiso})\Rightarrow (\ref{EquivRSequal})$. According to \cite[Proposition~3.6]{DubPol09}, the surface $Z_{P_r}$ and $Z_{P_s}$ are isomorphic if and only if there exist $a, \mu \in \k^\ast$, $\tau \in \k[t]$, such that 
	\[
		S_r(at, x) = \mu^2 S_s(t, \mu^{-1} x + \tau(t)) 
		\quad \textrm{inside $\k[t, x]$}.
	\] 
	This corresponds to
	\[(x-at)(x-1-a^2t^2r(at))=(x+\mu(\tau(t)-t))(x+\mu(\tau(t)-1-t^2s(t)))\] and thus gives two possibilities:
	
	$(\mathrm{I})$:  $at=\mu (t-\tau(t))$ and $1+a^2t^2r(at)=\mu(1+t^2s(t)-\tau(t))$. The first equation yields $\tau(t)=(1-\frac{a}{\mu}) t$ and the second yields $\mu\tau(t)\equiv \mu-1\pmod{t^2}$, which gives $\tau=0$ and then $\mu=1$ and $a=1$.
The second equation thus yields $r(t)=s(t)$.
	
	$(\mathrm{II})$: $at=\mu(1+t^2s(t)-\tau(t))$ and $\mu (t-\tau(t))=1+a^2t^2r(at)$. 	This yields 
	\[
		1+a^2t^2r(at) - \mu t =-\mu\tau(t)=at-\mu(1+t^2s(t))
	\]
	and thus $1-\mu t\equiv -\mu+at \pmod{t^2}$, whence $\mu = -1$ and $a = 1$. Replacing in the equation above, we find $r(t)=s(t)$.
\end{proof}

The proof of Theorem~$\ref{thm:Fibered_Embeddings}$ is now clear:
\begin{proof}[Proof of Theorem~$\ref{thm:Fibered_Embeddings}$]
Assertion~\eqref{ThmFibered1} corresponds to Proposition~\ref{Prop:LiftPAssertions}. 
	
Assertion~\eqref{ThmFibered2} follows from Corollary~\ref{Coro:HPHQequiv}. 
	
Assertion~\eqref{ThmFibered3} follows from Lemma~\ref{Lemm:InfiniteFamily}, which yields hypersurfaces $H_{P_r}\subset \SL_2$ that are 
parametrised by $r\in\k[t]$, which are all images of fibred embeddings and are pairwise non-equivalent.
	\end{proof}

We finish this subsection with two explicit examples:
\begin{lemma}\label{Lemm:NonEquivZ}
Let us denote by $P,Q\in\k[t, x, y]$ the polynomials
	\[
		P  = t^2y - x(x+1) \quad \textrm{and} \quad  
		Q = t^2 y - x(x+1 -t^2) \, .
	\]
	Then, the following hold:
	\begin{enumerate}[label=$(\arabic*)$, ref=\arabic*]
	\item\label{ZPQexequi}
	The hypersurfaces $Z_P,Z_Q\subset \A^3_\k$ given by $P=0$ and $Q=0$ are equivalent.
	\item\label{HPQexnotequi}
	The hypersurfaces $H_P,H_Q\subset \SL_2$ given by $P=0$ and $Q=0$ are both images of fibred embeddings but are not equivalent.
	\end{enumerate}
\end{lemma}
\begin{proof}
To get~\eqref{ZPQexequi}, it suffices to observe that the linear automorphism $\theta\in \Aut(\A^3_\k)$ given by $(t,x,y)\mapsto (t, x, y-x)$ satisfies $\theta^*(Q)=P$, so $\theta(Z_P)=Z_Q$.

Since $P,Q$ are of degree $1$ in $y$, both are variables of the $\k(t)$-algebra
$\k(t)[x,y]$. Moreover, $P(0,x,y)=Q(0,x,y)=-x(x+1)$ is of the  form  $\mu x^m(x-\lambda)$ (with $\mu,\lambda\in \k^*$ and $m=1\ge 0$). 
Since the coefficient of $y$ in $P$ and $Q$ is $t^2$, the morphisms 
$Z_{P},Z_Q \to \A^1_{\k}$, $(t, x, y) \mapsto t$ are trivial $\A^1$-bundles 
over $\A_{\k}^1\setminus \{0\}$. Proposition~\ref{Prop:LiftPAssertions} ($\eqref{ThmEqc}\Rightarrow\eqref{ThmEqb}$) then implies that $H_{P},H_Q\subset \SL_2$ are images of fibred embeddings $\A^2_\k\hookrightarrow \SL_2$.

To get~\eqref{HPQexnotequi}, we suppose that there is $\varphi\in \Aut(\SL_2)$ such that $\varphi(H_P)=H_Q$ and derive a contradiction. Corollary~\ref{Coro:HPHQequiv} yields an automorphism $\psi\in \Aut(\A^3_\k)$ such that $\psi^*(t)=\mu t$ for some $\mu\in \k^*$ and such that $\psi(Z_P)=Z_Q$ and $\psi(\Gamma)=\Gamma$, where $\Gamma\subset \A^3_\k$ is the conic given by $t=xy-1=0$. The restriction of $\psi$ to the hyperplane $H\subset \A^3_\k$ given by $t=0$ then preserves $\Gamma$ and also the curve $C=H\cap Z_P=H\cap Z_Q$, given by $t=x(x+1)=0$ (which is isomorphic to two copies of $\A^1$). The fact that $C$ is preserved implies that $\psi_{|_H}$ is of the form 
$(x, y) \mapsto (x,ay+p(x))$ or $(x,y)\mapsto (-1-x,ay+p(x))$ for some $a\in \k^*$ and $p\in \k[x]$. The fact that $\Gamma$ is preserved implies that $\psi_{|_H}=\mathrm{id}$.

The element $\xi=\theta^{-1}\psi\in \Aut(\A^3_\k)$ then satisfies $\xi(Z_P)=Z_P$,  $\xi^{*}(t)=\mu t$ and $\xi_{|_H}$ is the automorphism $(x,y)\mapsto (x,x+y)$. To show that this is impossible, we use \cite[Theorem~3.11]{DubPol09} 
	to see that every automorphism of 
	$Z_P$ preserves $C$ and its action on $C$ corresponds to an 
	element of the subgroup $G=G_0\cup G_1\simeq G_0\rtimes (\Z/2\Z)$ 
	of $\Aut(C)$ given by  
	\[
		\begin{array}{rcl}
		G_0&=&\{(x, y) \mapsto (x, \alpha y + (2x+1) \beta)\mid 
		\alpha\in \k^*,
		\beta\in \k\}\\
		G_1&=&\{(x, y) \mapsto (-1-x, \alpha y + (2x+1) \beta)
		\mid \alpha\in \k^*,
		\beta\in \k\}.
		\end{array}
		\qedhere
	\]
\end{proof}

We then study an explicit example of a fibred embedding $\A^2_{\k}\hookrightarrow \SL_2$ whose image is not equivalent to the standard embedding.
\begin{example}
According to the above study, the ``simplest'' example of a hypersurface 
$E\subset\SL_2$ being the image of a fibred embedding but not being 
equivalent to the image of the standard embedding is given by 
\[
	E=\{(t,u,x,y)\in \A^4_\k \mid xy-tu=1,ty=x(x-1)\}.
\]
Indeed, using the polynomial $P=ty-x(x-1)$, which yields $P(0,x,y)=-x(x-1)$, the surface $E$ is the image of a fibred embedding $\rho\colon \A^2_{\k}\hookrightarrow \SL_2$ (Example~\ref{LargeFamilySimpleExamples}) but is not equivalent to $\rho_1(\A^2_{\k})$ (Lemma~\ref{Lem:M0}).

One can construct an explicit embedding 
$\rho \colon \A^2_\k \hookrightarrow \SL_2$ having image $E$
in the following way. First, denoting by $E_{t}\subset E$ and $(\A^2_{\k})_{t} \subset \A^2_\k=\Spec(\k[x,t])$ the open subsets given by $t\not=0$, we get isomorphisms
\[\begin{array}{ccccccc}
 (\A^2_{\k})_{t}& \iso & E_t & \text{and} &  E_t & \iso & (\A^2_{\k})_{t}\\
 (x,t) & \mapsto & \begin{pmatrix}
				  x & t \\
				  \frac{x^2(x-1)-t}{t^2} & \frac{x(x-1)}{t}
			     \end{pmatrix} 
			     & & \begin{pmatrix}
				  x & t \\
				  u & y
			     \end{pmatrix}
			      &\mapsto & (x,t).
 \end{array}\]

To obtain a fibred embedding $\rho\colon \A^2_{\k} \hookrightarrow \SL_2$ having image equal to $E$, we need to remove the denominators of the isomorphism
$(\A^2_\k)_t \iso E_t$. 
We then compose with the automorphism of $(\A^2_{\k})_{t}$ given by 
$(x,t)\mapsto (t^2x+t+1,t)$ and get isomorphisms
\[
	\begin{array}{ccc}
		\A^2_{\k}& \stackrel{\rho}{\longrightarrow} & E \\
		(x,t) & \mapsto & 
			     \begin{pmatrix}
				  1+t+t^2x & t \\
				  \frac{(1+t+t^2x)^2(t+t^2x)-t}{t^2} & \frac{(1+t+t^2x)(t+t^2x)}{t}
			     \end{pmatrix}
	 \end{array}
\]
and
\[
 	\begin{array}{ccc}
		E & \stackrel{\rho^{-1}}{\longrightarrow}  & \A^2_{\k} \\
	 		    \begin{pmatrix}
				  x & t \\
				  u & y
			     \end{pmatrix}
			      &\mapsto & \left(\frac{x-t-1}{t^2},t\right) \, .
	\end{array}
\]
We can observe that all components of $\rho$ are indeed polynomials, and that $\frac{x-t-1}{t^2}\in \k[E]$. To show the latter, 
we compute 
$y=\frac{x(x-1)}{t}\in \k[E]$, $u=\frac{xy-1}{t}=\frac{x^2(x-1)-t}{t^2}\in \k[E]$, $y^2-ux+u=\frac{x-1}{t}\in \k[E]$, 
which yields $\frac{x-t-1}{t^2}=u-(x+1)\left(\frac{x-1}{t}\right)^2\in \k[E]$.

Writing 
\[
	\tilde\rho(x,t)=A\rho(x,t)A \quad  \text{with} \quad
	A=\left(\begin{array}{rr} 1 & 0 \\ -1 & 1\end{array}\right)\in \SL_2 \, ,
\] 
we get an equivalent closed embedding $ \tilde\rho\colon 
\A^2_{\k}\to \SL_2$, which is an isomorphism
\[\begin{array}{rccc}
 \tilde\rho\colon &\A^2_{\k}& \iso & \tilde{E}\\
& (x,t) & \mapsto & \begin{pmatrix}
				  1+t^2x & t \\
				  x+3tx+2t^2x^2+2t^3x^2+t^4x^3 & 1+tx+2t^2x+t^3x^2
			     \end{pmatrix},
 \end{array}\]
 where $\tilde{E}=\{(t,u,x,y)\in \A^4_\k \mid xy-tu=1,t(y+t)=(x+t)(x+t-1)\}.$ 
 The morphism $\tilde\rho$ corresponds to the 
closed embedding $\A^2_\k\hookrightarrow \A^4_\k$
 \[
 	\begin{array}{rllll}
	\A^2_{\k}& \hookrightarrow & \A^4_\k\\
	(x,t) & \mapsto & (t,1+t^2x, 1+tx+2t^2x+t^3x^2,
				  x+3tx+2t^2x^2+2t^3x^2+t^4x^3)\\
 	\end{array}
\]
that we can simplify using elementary automorphisms of $\A^4_\k$
to the embedding 
\[
	\begin{array}{rllll}		
	\A^2_{\k}& \hookrightarrow & \A^4_\k\\		  
	(x,t) & \mapsto & (t,t^2x, tx+t^3x^2,
				  x+2t^2x^2-t^3x^2+t^4x^3)\\
				  &=&(t,t^2x,tx(1+t^2x),x+t^2x^2(2-t+t^2x))
 	\end{array}
\]
\end{example}

\begin{question}
Is the closed embedding \[\begin{array}{rllll}
\A^2_{\k}& \hookrightarrow & \A^4_\k\\
(x,t) & \mapsto & (t,t^2x,tx(1+t^2x),x+t^2x^2(2-t+t^2x))
 \end{array}\]
 equivalent to the standard one?
\end{question}

\subsection{Embeddings of $\A^2_\k$ into $\SL_2$ of small degree}
This last subsection consists in showing the second part of Remark~\ref{Rem:HolSmall}, which claims that if all component functions of a closed
embedding $f \colon \A^2 \hookrightarrow \SL_2$ 
are polynomials of degree $\leq 2$, then $f$ is equivalent to $\rho_\lambda$
for a certain $\lambda \in \k^\ast$. This will be done in  Proposition~\ref{Prop:Low_degree} below, after a few lemmas.

 We first make the following easy observation:
\begin{lemma}\label{Lem:fibredembddingsst}
For each fibred embedding
\[
\begin{array}{cccc}
\rho\colon&\A^2_\k&\hookrightarrow& \SL_2\\
&(s,t)&\mapsto &\begin{pmatrix}
				  a(s,t) & t \\
				  c(s,t) & b(s,t)
			     \end{pmatrix}\end{array}\]
			     $($with $a,b,c\in \k[s,t])$
			      there is an automorphism $g\in \Aut(\SL_2)$ such that 
			      $g\rho$ is a fibred embedding given by \[
\begin{array}{cccc}
g\rho\colon&\A^2_\k&\hookrightarrow& \SL_2\\
&(s,t)&\mapsto &\begin{pmatrix}
				  1+stp(s,t) & t \\
				  s(p(s,t)+q(s,t)+stp(s,t)q(s,t)) & 1+stq(s,t)
			     \end{pmatrix}\end{array}\]
			     for some $p,q\in \k[s,t]$ such that $p(s,0)+q(s,0)\in \k^*$ and such that $\deg(1+stp(s,t))\le \deg(a)$, $\deg(1+stq(s,t))\le \deg(b)$ $($where the degree is here the degree of polynomials in $s,t)$.
\end{lemma}
\begin{remark}
The standard embedding $\rho_1$ is of the above form with $p=0$ and $q=1$. More generally, the embeddings $\{\rho_\lambda\}_{\lambda\in \k^*}$ of Theorem~\ref{Thm:StdEmbSL2} are given by $p=0$ and $q=\lambda$.
\end{remark}
\begin{proof}Replacing $t$ with $0$ yields two elements $a(s,0),b(s,0)\in \k[s]$ such that $a(s,0)\cdot b(s,0)=1$. This implies that $a(s,0),b(s,0)\in \k^*$. Applying the automorphism 
\[
				\begin{pmatrix}
				  x & t \\
				  u & y
			     \end{pmatrix}\mapsto \begin{pmatrix}
				  \mu x & t \\
				  u & \mu^{-1}y
			     \end{pmatrix}
\] 
for some $\mu\in \k^*$, we can assume that $a(s,0)=b(s,0)=1$. We then apply 
\[
				\begin{pmatrix}
				  x & t \\
				  u & y
			     \end{pmatrix}\mapsto \begin{pmatrix}
				  x & t \\
				  u & y
			     \end{pmatrix}\cdot \begin{pmatrix}
				  1 & 0 \\
				  d(t) & 1
			     \end{pmatrix}
%			     =\begin{pmatrix}
%				  x+d(t)t & t \\
%				  u+d(t)y & y
%			     \end{pmatrix}
\]
			     for some $d\in \k[t]$ and replace $a(s,t)$ with $a(s,t)+td(t)$, so can assume that $a(0,t)=1$. Applying similarly an automorphism of the form 
\[
			    \begin{pmatrix}
				  x & t \\
				  u & y
			     \end{pmatrix}\mapsto \begin{pmatrix}
				  1 & 0 \\
				  e(t) & 1
			     \end{pmatrix}\cdot \begin{pmatrix}
				  x & t \\
				  u & y
			     \end{pmatrix},
\] 
we can assume that $b(0,t)=1$. This yields $p,q\in \k[s,t]$ such that $a=1+stp$ and $b=1+stq$, which yields $c=s(p+q+stpq)$. Replacing $t$ with $0$ yields a closed embedding
			     \[
\begin{array}{cccc}
\A^1_\k&\hookrightarrow& \SL_2\\
s&\mapsto &\begin{pmatrix}
				  1 & 0 \\
				  s(p(s,0)+q(s,0)) & 1
			     \end{pmatrix}\end{array},\]whence $p(s,0)+q(s,0)\in \k^*$.
\end{proof}
\begin{corollary}\label{Cor:Fiberedembsmall}
Each fibred embedding
\[
\begin{array}{cccc}
\rho\colon&\A^2_\k&\hookrightarrow& \SL_2\\
&(s,t)&\mapsto &\begin{pmatrix}
				  a(s,t) & t \\
				  c(s,t) & b(s,t)
			     \end{pmatrix}\end{array}\]
			     where $a,b,c\in \k[s,t]$ are such that $\deg a+\deg b\le 4$ 
			     is equivalent to the embedding
\[
\begin{array}{cccc}
\rho_\lambda \colon &\A^2_\k&\hookrightarrow& \SL_2\\
&(s,t)&\mapsto &\begin{pmatrix}
				  1 & t \\
				  \lambda s & 1+\lambda st
			     \end{pmatrix}\end{array}
\]
for some $\lambda\in \k^*$.
\end{corollary}

\begin{proof}
Applying Lemma~\ref{Lem:fibredembddingsst}, one can assume that $a=1+stp$, $b=1+stq$, $c=s(p+q+stpq)$ for some $p,q\in \k[s,t]$ with $p(s,0)+q(s,0)\in \k^*$. If $p=0$, then  $\rho(\A^2)$ is equal to $\rho_1(\A^2)$, so the result follows from Theorem~\ref{Thm:StdEmbSL2}\eqref{Thm:StdEmbSL2ass1}. 
The same holds if $q=0$ by applying the automorphism
\[
			     \begin{pmatrix}
				  x & t \\
				  u & y
			     \end{pmatrix}\mapsto \begin{pmatrix}
				  y & t \\
				  u & x
			     \end{pmatrix} .
\] 
To finish the proof, we assume that $pq\not=0$ and derive a contradiction.
The fact that $\deg a+\deg b\le 4$ implies that $p,q\in \k^*$. Hence,
$\k[s, t] = \k[a,b,c,t]=\k[t,st,s(p+q+stpq)]=\k[t,st,s(st+\xi)]$ 
			     with $\xi=\frac{p+q}{pq} \neq 0$, and thus the morphism
			     \[\begin{array}{rcl}
			     \A^2_\k & \mapsto & \A^3_\k\\
			     (s,t) & \mapsto & (t,st,s(st+\xi))\end{array}\]
			     would be a closed embedding. This is false, since 
			     the image is properly contained in the irreducible 
			     hypersurface given by 
			     $\{(x,y,z)\in \A^3_\k\mid xz=y(y+\xi)\}$
			     (the line given by $x=y+\xi=0$ is missing).
			 \end{proof}

It remains to generalise Corollary~\ref{Cor:Fiberedembsmall} to the case of embeddings $\A^2_\k\hookrightarrow \SL_2$ of small degree (which are fibred or not).

In the sequel we will use the following subgroups of $\Aut(\A^2_\k)$:
\begin{definition}\label{Defi:Aff2GL2}
\[\begin{array}{lll}
\Aff_2(\k)&=&\left\{(s,t)\mapsto (as+bt+e,cs+dt+f)\left| \begin{pmatrix}
				  a & b \\
				  c & d
			     \end{pmatrix}\in \GL_2(\k),e,f\in \k\right\}\right.\vspace{0.1cm} \\ 
\GL_2(\k)&=&\left\{(s,t)\mapsto (as+bt,cs+dt)\left| 
			     a,b,c,d \in \k, \ ad-bc \neq 0 \right\}\right.\end{array}
			     \]
			     \end{definition}

\begin{lemma}\label{Lemma:GL2Aff2simpleform}
Let $\k$ be an algebraically closed field and let $\rho\colon \A^2_\k\to \A^2_\k\setminus \{0\}$ be a morphism of the form
\[(s,t)\mapsto (f(s,t),g(s,t))\]
such that $f,g$ have degree $2$ and that the homogeneous parts $f_2$ and $g_2$ of $f,g$ of degree~$2$ are linearly independent. Then, there exist $\alpha\in \Aff_2$ and $\beta\in \GL_2$ such that 
\[\beta\rho\alpha=(s,t)\mapsto (s^2,st+1).\]
\end{lemma}

\begin{proof}
We first observe that replacing $\rho$ with $\beta\rho\alpha$, where $\alpha\in \Aff_2$ and $\beta\in \GL_2$, does not change the degree of $f,g$ or the fact that $f_2$ and $g_2$ are linearly independent. We then observe that we can assume that $f_2=s^2$. If $f_2$ is a square, it suffices to  replace $f$ with $\rho\alpha$ for some $\alpha\in \GL_2$. If $f_2$ is not a square, we choose $\xi\in \k$ such that $g_2+\xi f_2$ is a square (this is possible since the discriminant of $g_2+\xi f_2$ is a polynomial of degree $2$ in $\xi$ and $\k$ is algebraically closed). 
We then apply an element of $\GL_2$ at the target to replace $f_2,g_2$ with $g_2+\xi f_2,f_2$,and then apply as before an element of $\GL_2$ at the source, to obtain $f_2=s^2$.

For each irreducible factor $P$ of $f$, we denote by $C_P\subset \A^2_\k=\Spec(\k[s,t])$ the irreducible curve given by $P=0$, and observe that $g$ yields an invertible function on $C_P$.

$(a)$ If $f$ is a product of factors of degree $1$, all belong to $\k[s]$, since $f_2=s^2$. We  can then write $f=\prod_{i=1}^2(s-\lambda_i)$, for some $\lambda_i\in \k$. If  $\lambda_1=\lambda_2$, we replace $s$ with $s-\lambda_1$ and get $f=s^2$, which yields  $g=s(a s+b t+c)+d$ for some $a,b,c\in \k,d\in \k^*$. The parts of degree $2$ of $f$ and $g$ being linearly independent, we get $b\not=0$. Replacing $t$ with $\frac{t-as-c}{b}$, we replace $g$ with $st+d$. We then apply diagonal elements of the form $(s, t)\mapsto (s, \mu t)$, 
$\mu \in k^\ast$ at the source and target and replace $d$ with $1$, which yields the desired form. To finish case $(a)$, it remains to see 
that $\lambda_1 \neq \lambda_2$ is impossible. To derive this contradiction, we apply an element of $\Aff_2$ at the source and get $f=s(s-1)$ This yields $g=sp(s,t)+\mu$, where $\mu\in \k^*$ and $p\in \k[s,t]$ is of degree $1$. We moreover obtain $p(1,t)\in \k\setminus \{-\mu\}$, so $p(s,t)=(s-1)\xi+\nu$ for some $\xi,\nu\in \k$. This yields $g\in \k[s]$, which is impossible since $g_2$ is not a multiple of $f_2=s^2$.

$(b)$ We can now assume that $f$ is not a product of factors of degree $1$, i.e.~$f$ is irreducible, and derive a contradiction. We observe that the curve $C_f\subset \A^2$ given by $f=0$ is isomorphic to $\A^1$. Indeed, the closure of $C_f$ in $\p^2_\k$ is an irreducible and thus a smooth conic with one point at infinity since $f_2=s^2$ (recall that $\k$ is assumed to be algebraically closed). 
This implies that the restriction $g |_{C_f}$ is a non-zero constant and so $g=\mu +\xi f$ for some $\mu\in \k^\ast,\xi\in \k$. This contradicts the fact that $f_2$ and $g_2$ are linearly independent.
\end{proof}

\begin{proposition}
\label{Prop:Low_degree}
Each closed embedding
\[
	\begin{array}{cccc}
		\rho\colon&\A^2_\k&\hookrightarrow& \SL_2\\
		&(s,t)&\mapsto &\begin{pmatrix}
				  f_{11}(s,t) & f_{12}(s,t) \\
				  f_{21}(s,t) & f_{22}(s,t)
			     \end{pmatrix}
	\end{array}
\]
where $f_{11},f_{12},f_{21},f_{22}\in \k[s,t]$ have at most degree $2$ is 
equivalent to the embedding
\[
\begin{array}{cccc}
\rho_\lambda \colon &\A^2_\k&\hookrightarrow& \SL_2\\
&(s,t)&\mapsto &\begin{pmatrix}
				  1 & t \\
				  \lambda s & 1+\lambda st
			     \end{pmatrix}\end{array}
\]
			     for some $\lambda\in \k^*$.
			     \end{proposition}
			     
			     \begin{proof}Applying 
			     Theorem~\ref{Thm:StdEmbSL2}, 
			     one only needs to show the existence of an automorphism of $\SL_2$ that sends $\rho(\A_{\k}^2)$ onto $\rho_1(\A_{\k}^2)$. We distinguish the following cases:
			     
			     $(a)$ Suppose first that one of the polynomials $f_{ij}$ is constant. One can assume that it is $f_{11}$ by using permutation of coordinates (with signs). The case $f_{11}=0$ is impossible, since the image would then be contained in 
\[
		\left.\left\{\begin{pmatrix}
				  x & t \\
				  u & y
		\end{pmatrix}\in \SL_2 \right| x=0\right\}
		\simeq (\A^1_{\k} \setminus \{0\})\times \A^1_{\k}.
\] 
We then have $f_{11}\not=0$ and apply a diagonal automorphism of $\SL_2$ to get $f_{11}=1$, which corresponds to $\rho(\A^2_{\k})=\rho_1(\A^2_{\k})$.
			     
$(b)$ Suppose then that one of the $f_{ij}$ has degree $1$. Applying permutations one can assume that $f_{12}$ has degree $1$.
Applying an element of $\Aff_2$ at the source (see Definition~\ref{Defi:Aff2GL2}), we do not change the degree of the ${f_{ij}}'s$ and can assume that $f_{12}=t$.  
Since $\deg(f_{11}f_{22})=\deg(f_{12}f_{21})\le 4$, the result follows from 
Corollary~\ref{Cor:Fiberedembsmall}.
			      
$(c)$ It remains to study the case where 
$\deg(f_{ij})= 2$ for each $i,j\in \{1,2\}$. If the homogeneous parts of $f_{11}$ and $f_{12}$ of degree $2$ are collinear, we apply 
\[	
		\begin{pmatrix}
				  x & t \\
				  u & y
		\end{pmatrix}\mapsto \begin{pmatrix}
				  x & t \\
				  u & y
		\end{pmatrix}\begin{pmatrix}
				  1 & \mu \\
				  0 & 1
		\end{pmatrix}=\begin{pmatrix}
				  x & t+\mu x \\
				  u & y+\mu u
		\end{pmatrix}
\] 
for some $\mu\in \k$ and obtain $\deg(f_{12})\le 1$, which reduces to the cases $(a)$, $(b)$. To achieve the proof of $(c)$, we now assume that the homogeneous parts of $f_{11}$ and $f_{12}$ of degree $2$ are linearly independant and prove that this implies that $\k[f_{11},f_{12},f_{21},f_{22}]\subsetneq \k[s,t]$ (which contradicts the fact that $\rho$ is a closed embedding). To show this, one can extend the scalars and assume that $\k$ is algebraically closed. We then apply Lemma~\ref{Lemma:GL2Aff2simpleform} to the morphism $\nu\colon \A^2_\k\to \A^2_\k\setminus \{0\}$ given by $(s,t)\mapsto (f_{11}(s,t),f_{12}(s,t))$, and find $\alpha\in \Aff_2(\k)$, $\beta\in \GL_2(\k)$ such that $\beta\nu\alpha=(s,t)\mapsto (s^2,st+1)$. We write $\mu=\det(\beta)\in \k^*$ and replace $\rho$ with $\hat\beta \rho\alpha$, where $\hat\beta\in \Aut(\SL_2)$ is of the form
\[
		\hat\beta\colon
		\begin{pmatrix}
				  x & t \\
				  u & y
		\end{pmatrix}\mapsto \begin{pmatrix}
				  1 & 0 \\
				  0 & \mu^{-1}
		\end{pmatrix}\cdot \beta\cdot\begin{pmatrix}
				  x & t \\
				  u & y
		\end{pmatrix}. 
\]
This change being made, we obtain $f_{11}=s^2$, $f_{12}=st+1$. Since $1=f_{11}f_{22}-f_{12}f_{21}=s^2f_{22}-(st+1)f_{21}$, we obtain $f_{21}=st-1 + g(s,t) s^2$ for some $g\in \k[s,t]$. This implies that $f_{11},f_{12}-1,f_{21}+1$ all belong to the maximal ideal $(s,t)^2\subset \k[s,t]$, which yields the desired contradiction  $\k[f_{11},f_{12},f_{21},f_{22}]=\k[f_{11},f_{12}-1,f_{21}+1,f_{22}]\subsetneq\k[s,t]$.    
\end{proof}
		
\section{A non-trivial embedding of $\A^1$ into $\SL_2$, over the reals}
\label{Sec:Realexample}
In this section, we provide over the field $\k = \R$ 
an explicit example of an algebraic embedding 
$\A^1_{\R} \hookrightarrow \SL_2$ which is not equivalent 
to the standard embedding
\[
	\begin{array}{cccc}
	\tau_1 \colon & \A^1_\R & \hookrightarrow & \SL_2 \\
	& t & \mapsto & \begin{pmatrix}
				1 & 0 \\
				t & 1
			\end{pmatrix} \, .
	\end{array}
\]
\begin{example}
\label{Ex:Shastri}
In \cite{Shastri92} the closed embedding
\[
		\begin{array}{cccc}
	\gamma \colon & \A^1 & \hookrightarrow & \A^3 \\
	& t & \mapsto & (t^3-3t, t^4-4t^2-1, t^5-10t)
	\end{array}
\]
is given. This one is not equivalent to the standard embedding 
$\A^1 \hookrightarrow \A^3$, $t \mapsto (t, 0, 0)$, over the field $\R$ of real numbers. 
The reason is that it corresponds, as an embedding $\R \hookrightarrow \R^3$, to the (open) trefoil knot.

The fact that $\gamma$ is a closed embedding, over any field $\k$, can be shown as follows. Writing $\gamma_1=t^3-3t,
 \gamma_2=t^4-4t^2-1, \gamma_3=t^5-10t \in \k[t]$, we get
\[t=3\gamma_3-12\gamma_1-5\gamma_1\gamma_2+\gamma_2\gamma_3-\gamma_1^3.\]

The fact that $\gamma\colon \R\to \R^3$ corresponds to the open trefoil knot can be seen by looking at the three projections:
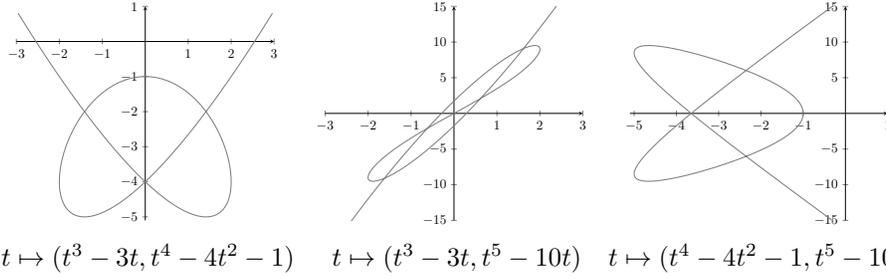
\begin{center}\begin{tikzpicture}[scale=0.50]
\begin{axis}[axis lines=middle, xmin=-3, xmax = 3, ymin=-5.1, ymax = 1,xtick={-3,-2,...,3},
  ytick={-5,...,1}]
    \addplot[My Line Style, color=gray,  variable=\t, domain=-2.1:2.1]({\t^3-3*\t},{\t^4-4*\t^2-1});
\end{axis}
\node at (3.5,-1) {$t\mapsto (t^3-3t,t^4-4t^2-1)$};
\end{tikzpicture}
\begin{tikzpicture}[scale=0.50]
\begin{axis}[axis lines=middle, xmin=-3, xmax = 3, ymin=-15, ymax = 15,xtick={-3,-2,...,3},
  ytick={-15,-10,-5,...,15}]
    \addplot[My Line Style, color=gray,  variable=\t, domain=-2.1:2.1]({\t^3-3*\t},{\t^5-10*\t});
\end{axis}
\node at (3.5,-1) {$t\mapsto (t^3-3t,t^5-10t)$};
\end{tikzpicture}
\begin{tikzpicture}[scale=0.50]
\begin{axis}[axis lines=middle, xmin=-5.1, xmax = 1, ymin=-15, ymax = 15,xtick={-5,...,1},
  ytick={-15,-10,-5,...,15}]
    \addplot[My Line Style, color=gray,  variable=\t, domain=-2.1:2.1]({\t^4-4*\t^2-1},{\t^5-10*\t});
\end{axis}
\node at (3.5,-1) {$t\mapsto (t^4-4t^2-1,t^5-10t)$};
\end{tikzpicture}\end{center}
\end{example}

We now use Example~\ref{Ex:Shastri} to provide a similar example in $\SL_2$:
\begin{lemma}\label{Exam:RSL2R}\item
\begin{enumerate}[label=$(\arabic*)$, ref=\arabic*]
\item \label{ClosedEmb}
For each field $\k$ of characteristic $\not=2$, the morphism 
\[\begin{array}{cccc}
\tau\colon &\A^1 & \hookrightarrow & \SL_2\\
&t& \mapsto & \begin{pmatrix}
			t^3-3t & t^4-4t^2-1\\
			1+ \frac{t^2(17t^6-56t^4-137t^2+452)}{16}& 
			\frac{t(17t^8-73t^6-149t^4+609t^2+172)}{16}
		     \end{pmatrix}
		     \end{array}\]
		     is a closed embedding.
		     \item\label{kRnotstd}
		     If $\k=\R$, then $\tau$ is not equivalent to the standard embedding, because the fundamental group $\pi_1(\SL_2(\R)\setminus \tau(\R))$ is not isomorphic to the free group $\pi_1(\SL_2(\R)\setminus \tau_1(\R))$.\end{enumerate}
\end{lemma}
\begin{proof}
\eqref{ClosedEmb}: The fact that $\tau$ is a closed embedding can be done explicitely by giving a formula for $t$, but can also be shown by using the $\A^1$-bundle
\[\begin{array}{cccc}
	p\colon &\SL_2 &  \to & \A^2 \setminus \{ (0, 0)\} \\
	  &\begin{pmatrix}
		x & t \\
		u & y
	   \end{pmatrix}
	& \mapsto & (x, t).
	\end{array}\]
	
Writing $\gamma_1=t^3-3t,
 \gamma_2=t^4-4t^2-1, \gamma_3=t^5-10t \in \k[t]$ as in Example~\ref{Ex:Shastri}, we get
 $\gamma_1^2(\gamma_1^2-4)-\gamma_2(\gamma_2^2+9\gamma_2+24)=16$ and thus get a birational morphism 
 \[\begin{array}{ccc}
 \A^1&\to &\Gamma = \{(x,t)\in \A^2\mid x^2(x^2-4)-t(t^2+9t+24)=16\} \\
 t & \mapsto & (\gamma_1(t),\gamma_2(t))\end{array}\]
from $\A^1$ to the singular affine quartic curve $\Gamma\subset \A^2$. We then get a morphism
\[\begin{array}{rccc}
	f\colon & \Gamma&\to&  \SL_2\\
	&(x,t) & \to & \begin{pmatrix}
		x & t \\
		\frac{t^2+9t+24}{16} & \frac{x(x^2-4)}{16}
	   \end{pmatrix}
	\end{array}\]
	which satisfies $p \circ f=\mathrm{id}_\Gamma$ and is thus a section of 
	$p$ over $\Gamma$. This implies that
\[
	\begin{array}{ccc}
	\Gamma \times \A^1 & \hookrightarrow & \SL_2 \\
	((x, t), a)& \mapsto & \begin{pmatrix}
						1 & 0 \\
						a & 1
				   	\end{pmatrix}f(x, t) 
	\end{array}
\]
is a closed embedding. Since $\gamma\colon \A^1\to \Gamma\times \A^1\subset \A^3$ is a closed embedding,
the morphism
\[
	\begin{array}{cccc}
	\tau \colon & \A^1 & \hookrightarrow & \SL_2 \\
	& t & \mapsto & \begin{pmatrix}
						1 & 0\\
						t^5-10t  & 1
				   	\end{pmatrix}f(\gamma_1(t),\gamma_2(t))
	\end{array}
\]
is a closed embedding. Replacing $\gamma_1$ and $\gamma_2$ in the above formula yields the explicit form of the morphism given in the statement of the lemma.
	
	\eqref{kRnotstd}: In the remaining part of the proof, we work over $\k=\R$ and use the Euclidean topology. The $\R$-bundle $p\colon \SL_2(\R)\to \R^2\setminus \{(0,0)\}$ is trivial, as it admits a (rational) continuous section given by 
	\[\begin{array}{cccc}
	\xi \colon & \R^2\setminus \{(0,0)\} & \to & \SL_2(\R)\\
	& (x,t)&\mapsto &\begin{pmatrix}
		x & t \\
		-\frac{t}{x^2+t^2} & \frac{x}{x^2+t^2}
	   \end{pmatrix}.
	\end{array}\]
	This yields a birational diffeomorphism
	\[
	\begin{array}{cccc}
		\varphi \colon &\R^2\setminus \{(0,0)\}\times \R & \to & \SL_2(\R) \\
		&((x, t), a) & \mapsto & \begin{pmatrix}
							1 & 0 \\
							a & 1 
						\end{pmatrix}
						\xi(x,t) \, .
	\end{array}
	\] 
	
	In particular, $\SL_2(\R)\setminus \tau_1(\R)$ is diffeomorphic to $\R^2\setminus \{(0,0),(0,1)\}\times \R$, which implies that the fundamental group $\pi_1(\SL_2(\R)\setminus \tau_1(\R))$ is a free group (over two generators). It remains to show that $\pi_1(\SL_2(\R)\setminus \tau(\R))$ is not a free group. This will imply that no diffeomorphism of $\SL_2(\R)$ sends $\tau(\R)$ onto $\tau_1(\R)$, and in particular no algebraic automorphism defined over $\R$.
	
We extend $f\colon \Gamma(\R)\to \SL_2(\R)$ to a global continuous section 
$\hat f\colon \R^2\setminus \{(0,0)\}\to \SL_2(\R)$ of $p$ 
(which exists, since
$p$ is a trivial $\R$-bundle). 
This yields a rational diffeomorphism
\[
	\begin{array}{ccc}
	g\colon \R^2\setminus \{(0,0)\} \times \R & \iso & \SL_2(\R) \\
	((x, t), a)& \mapsto & \begin{pmatrix}
						1 & 0 \\
						a & 1
				   	\end{pmatrix}\hat{f}(x, t) 
	\end{array}
\]
which maps $\gamma(\R)$ onto $\tau(\R)$.

	We take an open subset $U \subset \R^2 \setminus \{ (0, 0) \}$ (for the Euclidean topology) that contains the singular curve $\Gamma(\R)$ and a  homeomorphism
$h \colon U \iso \R^2$ which fixes $\Gamma(\R)$ pointwise, and which is homotopic to the inclusion $U\hookrightarrow \R^2$, via a homotopy that fixes $\Gamma(\R)$ pointwise. We can for instance take $U=\{(x,t)\in \R^2\mid t\le x^2 - \frac{1}{2}\}$ and construct a homeomorphism and a homotopy 
which preserve the fibres of the projection $(x,t)\mapsto x$. 
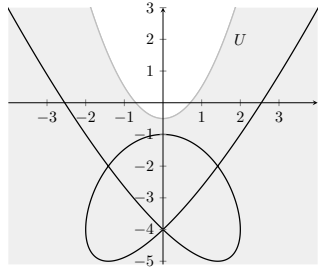
\begin{center}\begin{tikzpicture}[scale=0.60]
\begin{axis}[axis lines=middle, xmin=-4, xmax = 4, ymin=-5.1, ymax = 3,xtick={-3,-2,...,3},
  ytick={-5,...,3},axis on top]
    \addplot[My Line Style, color=lightgray,  variable=\t, domain=-4:4]({\t},{\t^2-0.5});\addplot [
            My Line Style,
            draw=none,
            fill=lightgray!25,
            domain=-4:4,
            % restrict the y values to the maximum y axis value
            restrict y to domain*=-5.1:3,
        ] {x^2-0.5}
            % continue the path to close it
            |- (\pgfkeysvalueof{/pgfplots/xmin},\pgfkeysvalueof{/pgfplots/ymin})
        ;
        % and on top of that the line
        \addplot [
            My Line Style,
            draw=lightgray,
            domain=-4:4,
        ] {x^2-0.5}
        ;
    \addplot[My Line Style, color=black,  variable=\t, domain=-4:4]
    ({\t^3-3*\t},	{\t^4-4*\t^2-1});
    \node[] at (axis cs: 2,2) {$U$};
\end{axis}
\end{tikzpicture}\end{center}The homeomorphisms
\begin{equation}
	\label{Eq:Homeom}
	\tag{$\star$}
	p^{-1}(U) \setminus \tau(\R) 
	\stackrel{g^{-1}}{\longrightarrow} 
	\left(U \times \R\right) \setminus \gamma(\R)
	\stackrel{ h\times \mathrm{id}}{\longrightarrow} \R^3 \setminus \gamma(\R)
\end{equation}
yield isomorphisms of the fundamental groups
\[
	\pi_1(p^{-1}(U) \setminus \tau(\R)) \iso 
	\pi_1((U \times \R) \setminus \gamma(\R)) \iso 
	\pi_1(\R^3 \setminus \gamma(\R)) \, .
\]

Since $\gamma \colon \R \hookrightarrow \R^3$ is the (open) trefoil knot, 
it follows that $\pi_1(\R^3 \setminus \gamma(\R))$ is the braid group
with $3$ strands and thus $\pi_1(p^{-1}(U) \setminus \tau(\R))$ is not a free 
group. It remains then to see that the group homomorphism
\[
	\iota\colon \pi_1(p^{-1}(U) \setminus \tau(\R)) \to \pi_1(\SL_2(\R) 
	\setminus \tau(\R)) 
\]
induced by the inclusion 
$p^{-1}(U) \setminus \tau(\R)\hookrightarrow \SL_2(\R)\setminus \tau(\R)$ 
is injective (as a subgroup of a free group is free).
Every element $\alpha\in \mathrm{Ker}(\iota)$ lies in the kernel of the map 
$\iota'\colon\pi_1(p^{-1}(U)\setminus \tau(\R))\to \pi_1(\R^3 \setminus \gamma(\R))$
induced by the composition 
\[
	p^{-1}(U) \setminus \tau(\R) \hookrightarrow \SL_2(\R) \setminus \tau(\R)
	\stackrel{g^{-1}}{\longrightarrow} 
	((\R^2 \setminus \{ (0,0) \}) \times \R) \setminus \gamma(\R)
	\hookrightarrow \R^3 \setminus \gamma(\R),
\]
which corresponds simply to the composition
\begin{equation}
	\label{Eq:InclusHom}
	\tag{$\star\star$}
	p^{-1}(U) \setminus \tau(\R) 
	\stackrel{g^{-1}}{\longrightarrow} 
	\left(U \times \R\right) \setminus \gamma(\R)
	\hookrightarrow \R^3 \setminus \gamma(\R).
\end{equation}

Since $h\colon U\to \R^2$ is homotopic to the inclusion $U\hookrightarrow \R^2$
via a homotopy that fixes $\Gamma(\R)$ pointwise, the two compositions $p^{-1}(U) \setminus \tau(\R)\to \R^3 \setminus \gamma(\R)$ of \eqref{Eq:Homeom} and \eqref{Eq:InclusHom} are homotopic, so $\iota'$ is an isomorphism. This implies that $\iota$ is injective and achieves the proof.
\end{proof}

\begin{question}
	Working over the field of complex numbers $\C$, is the algebraic embedding 
	$\tau\colon \A^1\to \SL_2$ of 
	Lemma~$\ref{Exam:RSL2R}\eqref{ClosedEmb}$ 
	equivalent to the standard embedding 
	$\tau_1 \colon \A^1 \hookrightarrow \SL_2$?
\end{question}

\end{document}